\documentclass[a4paper,10pt]{article}
\usepackage[english]{babel}
\usepackage{amsmath,amssymb,xcolor,latexsym,theorem}
\usepackage[left= 2.6cm, bottom = 4cm, top = 3.5cm, right= 2.6cm]{geometry}
\usepackage[citecolor=blue,colorlinks=true]{hyperref}
\usepackage{authblk}

\newcommand{\assign}{:=}
\newcommand{\backassign}{=:}
\newcommand{\cdummy}{\cdot}

\newcommand{\nobracket}{}
\newcommand{\nosymbol}{}
\newcommand{\point}{.}

\newcommand{\tmmathbf}[1]{\ensuremath{\boldsymbol{#1}}}

\newcommand{\tmop}[1]{\ensuremath{\operatorname{#1}}}
\newcommand{\tmtextbf}[1]{\text{{\bfseries{#1}}}}
\newcommand{\tmtextit}[1]{\text{{\itshape{#1}}}}
\newenvironment{proof}{\noindent\textbf{Proof\ }}{\hspace*{\fill}$\Box$\medskip}
\newtheorem{theorem}{Theorem}[section]
\newtheorem{corollary}[theorem]{Corollary}
\newtheorem{definition}[theorem]{Definition}
\newtheorem{proposition}[theorem]{Proposition}
{\theorembodyfont{\rmfamily}\newtheorem{remark}[theorem]{Remark}}

\def\div{\mathord{{\rm div}}}

\newcommand{\rmb}[1]{\textcolor{black}{#1}}

\numberwithin{equation}{section}

\begin{document}

\title{Rough analysis of two scale systems}

\author[1]{Arnaud Debussche}
\author[2]{Martina Hofmanov{\'a}}
\affil[1]{\small Univ Rennes, CNRS, IRMAR - UMR 6625, F-35000 Rennes, France. \href{mailto:arnaud.debussche@ens-rennes.fr}{arnaud.debussche@ens-rennes.fr}}
  \affil[2]{\small Fakult\"at f\"ur Mathematik, Universit\"at Bielefeld, Postfach 10 01 31, D-33501 Bielefeld, Germany. \href{mailto:hofmanova@math.uni-bielefeld.de}{hofmanova@math.uni-bielefeld.de}}
\maketitle

\begin{abstract}
 We address  a slow-fast system of coupled three dimensional
  Navier--Stokes equations where the fast component is perturbed by an
  additive Brownian noise. By means of the rough path theory, we establish the
  convergence in law of the slow component towards a Navier--Stokes system
  with an It{\^o}--Stokes drift and a rough path driven transport noise. This
  gives an alternative, more general and direct proof to {\cite{DP22}}. Notably, the limiting rough path is identified as a geometric rough
  path, which does not necessarily coincide with the Stratonovich lift of the
  Brownian motion.
  
  \tmtextbf{Keywords:} Navier--Stokes equations, transport noise,
  It{\^o}--Stokes drift, rough paths
\end{abstract}

\

{\tableofcontents}

\section{Introduction}

Fluid models have numerous applications, in particular in the modeling of climate, oceans, atmosphere etc. These systems are often described mathematically by a system of partial differential equations. Increasingly,  we
recognize the importance of random effects influencing these phenomena and the necessity of incorporating such effects for both qualitative and quantitative predictions. 

In \cite{BCCLM20,M14}, in a similar spirit as \cite{MR04}, the classical derivation of fluid models from Reynold Transport Theorem has been revisited. A stochastic component was added to the velocity field to represent small scales that are challenging to account for in numerical simulations, leading to the derivation of a Stochastic Reynolds Transport Theorem.  Applying this to  basic physical balance laws yields stochastic versions of fluid equations, known as Location Uncertainty (LU) models. Interestingly, these equations contain a transport noise and an additional advection term induced by  the noise, referred to as the It\^o-Stokes drift. Other derivations of stochastic fluid equations have also been proposed. For instance, in \cite{H15}, a stochastic velocity component was introduced, and a variational principle was employed to  derive a stochastic version of the Euler equations as critical points of a stochastic energy. This approach, known as the SALT model -- Stochastic Advection by Lie Transport --  also incorporates  transport noise but differs from LU models in that it preserves more geometric quantities. Importantly, in both models, the noise is interpreted in the Stratonovich sense. 

From a mathematical perspective, the stochastic Navier--Stokes equations with transport noise represent a timely and active area of research. These equations were first studied in the seminal work \cite{BrCaFl92} and have recently gained further attention. A significant discovery was made in \cite{Ga20}, where it was shown that one can construct sequences of noise covariance operators such that the noise vanishes in the limit, while the It\^o--Stratonovich correction persists. The limiting deterministic equation retains a trace of the vanishing noise, manifesting as an additional elliptic operator termed enhanced viscosity. This phenomenon has been further explored in \cite{FlGaLu21c,FlLu21}.

On another front, multiscale approaches have been developed to model   small and large scales differently. For instance, in \cite{MTE01}, small-scale terms in the equations were neglected and replaced by noise, though this analysis was restricted to finite-dimensional models. Inspired by this, \cite{FP21, FP22} examined a multiscale system of Euler equations, proving convergence to the Euler equations with transport noise. Similarly, in this work, we study a slow-fast system parameterized by a small scaling parameter $\varepsilon \in (0, 1)$:
\begin{equation}
  \partial_t u^{\varepsilon} = A u^{\varepsilon} + b (u^{\varepsilon} +
  v^{\varepsilon}, u^{\varepsilon}), \label{eq:uintro}
\end{equation}
\begin{equation}
  d v^{\varepsilon} = A v^{\varepsilon} d t + b (u^{\varepsilon} +
  v^{\varepsilon}, v^{\varepsilon}) d t + \varepsilon^{- 1} C v^{\varepsilon}
  d t + \varepsilon^{- 1} Q^{1 / 2} d W, \label{eq:v}
\end{equation}
where $A$, $C$ are linear operators, $b$ a nonlinear term, and the  second equation is driven by a cylindrical Wiener process $W$ on some probability space $(\Omega, \mathcal{F},
\tmmathbf{P})$, with a  covariance operator $Q$.  Additional assumptions on the operators and parameters will be presented in Section \ref{s:21}.

Our primary goal is to rigorously analyze the asymptotic behavior of system \eqref{eq:uintro},
	\eqref{eq:v} as $\varepsilon \to 0$, focusing on the interactions between the large- and small-scale components under stochastic perturbations. In particular, we derive the limit of \eqref{eq:uintro}
as a stochastic dynamics driven by a certain transport noise, along with the It\^o--Stokes drift. The limiting behavior we obtain emphasizes the role of stochastic transport noise as an effective model reduction for systems where large-scale dynamics are influenced by fast, highly oscillatory components.  

A key example  is the coupled 
Navier--Stokes system:
\begin{equation}\label{eq:i1}
  \partial_t u^{\varepsilon} + (u^{\varepsilon} +
  v^{\varepsilon})\cdot\nabla u^{\varepsilon} +\nabla p^{\varepsilon}= \nu\Delta u^{\varepsilon} ,\qquad \mbox{div}\, u^\varepsilon = 0,
\end{equation}
\begin{equation}\label{eq:i2}
  d v^{\varepsilon} + (u^{\varepsilon} +
  v^{\varepsilon})\cdot\nabla v^{\varepsilon} d t +\nabla q^{\varepsilon} dt = \nu\Delta v^{\varepsilon} d t + \varepsilon^{- 1} C v^{\varepsilon}
  d t + \varepsilon^{- 1} d \mathcal{W},\qquad \mbox{div}\, v^\varepsilon = 0,
\end{equation}
where $t \in [0,T]$, $u^\varepsilon$ and $ v^\varepsilon$ represent velocity fields on a bounded domain $\mathcal O\subset \mathbb R^3$. The pressure fields $p^\varepsilon$ and $q^\varepsilon$ ensure the validity of  the divergence-free conditions $\mbox{div}\, u^\epsilon=\mbox{div}\, v^\epsilon =0$, and $\mathcal{W}$ is a Wiener process on $[L^2(\mathcal O)]^3$. The parameter $\varepsilon \in (0, 1)$ denotes a small scaling, and  $\nu>0$ is the viscosity. These equations are supplemented with boundary conditions which can be Dirichlet or, if $\mathcal O$ is the torus, periodic. For simplicity,   we consider  periodic boundary conditions.
Therefore, \eqref{eq:i1}, \eqref{eq:i2} is of the form \eqref{eq:uintro}, \eqref{eq:v} provided  $A$ is the Stokes operator, $b$ is the Navier--Stokes nonlinearity, and $Q^{1 / 2}$ is the covariance operator of the divergence-free part of $\mathcal{W}$. We refer to Section \ref{s:21} for further details.

This system  represents a decomposition of the
Navier--Stokes equations into a slow-varying, large-scale component
$u^{\varepsilon}$ and a fast-varying, small-scale component
$v^{\varepsilon}$. The fast component is driven by a  large additive stochastic
noise and, for the particularly relevant case of $C = - \tmop{Id}$, it experiences a large friction, which can also be interpreted as a scale separation. The noise amplitude is chosen such  that, formally, as $\varepsilon\to0$, the fast component $v^\varepsilon$ converges to the white noise $Q^{1/2}dW$. Consequently,  the term $(v^\varepsilon \cdot \nabla) u^\varepsilon$ in the large-scale equation \eqref{eq:i1} is expected to converge to a transport noise. However, as we shall demonstrate, this formal reasoning is incomplete due to the emergence of  the It\^o--Stokes drift which arises in the limiting system.

In \cite{FP21, FP22}, the non viscous case -- $\nu=0$ -- was studied and the limit $\varepsilon\to 0$  established thanks to a Wong-Zakai type argument.   It has to be noted that in \cite{FP21, FP22}, the It\^o--Stokes drift is not present due to an extra assumption on the noise needed in their argument.

The asymptotic behavior of the system \eqref{eq:uintro}, \eqref{eq:v} as $\varepsilon \rightarrow 0$
was recently studied  in {\cite{DP22}}.
The analysis in 
{\cite{DP22}} is based  on the perturbed test function method
introduced in {\cite{PaStVa77}}, which has since been widely used in the literature. The approach of {\cite{DP22}} relies on several assumptions: certain
regularity of the operator $Q$, symmetry and certain smoothing properties of $C$ (as e.g. $C = - (- \Delta)^{\gamma / 2}$ for some $\gamma > 1
/ 4$), and commutativity of $C$ and $Q$. Under these conditions, they prove that 
$u^{\varepsilon}$ converges in law to a solution of
\[ d u = A u d t + b (u, u) d t + b \left( \overline{r}, u \right) d t + b ((-
   C)^{- 1} Q^{1 / 2} \circ d W, u), \]
where $\bar{r}$ denotes  the It{\^o}--Stokes velocity as given in
\eqref{eq:itostokes} below. Particularly, the stochastic integral is
understood in Stratonovich's sense.

In contrast to the approach in \cite{DP22}, our work introduces a more direct method to address the limiting behavior as $\varepsilon\to0$. This approach allows us to significantly relax the assumptions on the operator $C$, encompassing physically relevant cases such as $C=-\tmop{Id}$.
 Moreover, we do not require $C$ to be symmetric, nor do we assume it commutes with $Q$. Rather surprisingly, in such situations, the stochastic
integral in the limit equation is no longer Stratonovich; instead, a non-trivial finite-variation perturbation arises. It is defined explicitly in terms of the operators $C$ and $Q$. We interpret the resulting stochastic integral within the framework of rough paths.
More specifically, our analysis is grounded in the theory of unbounded rough drivers, introduced in \cite{BG17} and further developed in \cite{DGHT19, FHLN20, MR3918521, HLN21}. We extend this theory to accommodate infinite-dimensional rough paths, which may have broader applications. As is typical in rough path theory, our approach imposes stricter regularity conditions on $Q$ compared to purely probabilistic methods such as those in \cite{DP22}.

In a parallel development, certain concepts from the present article were integrated with those from \cite{DP22} and further advanced to establish anomalous and total dissipation in passive scalar equations and Navier--Stokes equations advected by solutions to the randomly forced Navier--Stokes equations, as presented in \cite{HPZZ23b}.

\subsection{Summary of main  results and ideas}

We now provide a brief outline of our method. Unlike the approach taken in {\cite{DP22}}, we introduce a decomposition of the form $v^{\varepsilon} = \varepsilon^{- 1 / 2}
w^{\varepsilon} + r^{\varepsilon}$ where the components evolve according to the following system of equations:
\begin{equation}
  \partial_t u^{\varepsilon} = A u^{\varepsilon} + b (u^{\varepsilon} +
  \varepsilon^{- 1 / 2} w^{\varepsilon} + r^{\varepsilon}, u^{\varepsilon}),
  \label{eq:u}
\end{equation}
\begin{equation}
  d w^{\varepsilon} = \varepsilon^{- 1} C w^{\varepsilon} d t + \varepsilon^{-
  1 / 2} Q^{1 / 2} d W, \qquad w^{\varepsilon}_0 = 0, \label{eq:w}
\end{equation}
\begin{equation}
  \partial_t r^{\varepsilon} = \varepsilon^{- 1} C r^{\varepsilon} + A
  (\varepsilon^{- 1 / 2} w^{\varepsilon} + r^{\varepsilon}) + b
  (u^{\varepsilon} + \varepsilon^{- 1 / 2} w^{\varepsilon} + r^{\varepsilon},
  \varepsilon^{- 1 / 2} w^{\varepsilon} + r^{\varepsilon}), \qquad
  r^{\varepsilon}_0 = v^{\varepsilon}_0 . \label{eq:r}
\end{equation}

We do not present the details of the construction of solutions for this system at each fixed level of  $\varepsilon \in (0, 1)$, as these can be derived using classical techniques similar to those in \cite{DP22}. Specifically, the rescaled Ornstein--Uhlenbeck process  $w^{\varepsilon}$ exists on any
probability space that includes a $Q$-Wiener process $Q^{1 / 2} W$, and is uniquely determined and measurable with respect to $Q^{1 / 2} W$. Moreover, it possesses a unique invariant measure $\mu$, which is  independent of $\varepsilon$. See \cite{DPZ96,DPZa14} for further details.

Additionally, for each fixed $\varepsilon \in (0, 1)$ there exists a probabilistically and analytically weak solution to the  system \eqref{eq:u}, \eqref{eq:w}, \eqref{eq:r}. In other words, for every $\varepsilon \in (0, 1)$
there exists a probability space $(\Omega^{\varepsilon},
\mathcal{F}^{\varepsilon}, \tmmathbf{P}^{\varepsilon})$, a $Q$-Wiener process
$Q^{1 / 2} W^{\varepsilon}$ and processes $u^{\varepsilon}, w^{\varepsilon},
r^{\varepsilon}$ that solve \eqref{eq:u}, \eqref{eq:w}, \eqref{eq:r} in the
analytically weak sense. The proof of existence is based on a Galerkin approximation combined with the method of stochastic compactness, as was employed, for instance, in \cite{FG95} for the stochastic Navier--Stokes equations.

It is important to note that in the framework of probabilistically weak solutions, both the probability space and the Wiener process form part of the solution itself. This distinguishes probabilistically weak solutions from probabilistically strong solutions, where the probabilistic elements are specified in advance. A typical example of the latter is the Ornstein--Uhlenbeck process defined by \eqref{eq:w}, as discussed above. In general, the existence of probabilistically strong solutions can only be ensured for problems where uniqueness holds. One notable exception is the method of convex integration, which allows for the construction of non-unique solutions in various fluid dynamics models (see, for instance, \cite{HZZ19}). However, for the system \eqref{eq:u}, \eqref{eq:w}, \eqref{eq:r}, the question of uniqueness remains unresolved and may not hold, as suggested by recent studies \cite{ABC22,BV19a,HZZ19}. This uncertainty regarding uniqueness is the primary reason we focus our analysis on probabilistically weak solutions.

With the existence of solutions for each fixed level  $\varepsilon$, our next objective is to reformulate the system \eqref{eq:u}, \eqref{eq:w}, \eqref{eq:r} in the framework of rough path theory, enabling us to derive estimates that are uniform in $\varepsilon$. This step is detailed in Section~\ref{s:rr}, following the introduction of the basic concepts of rough path theory in Section~\ref{s:RP}. For a detailed introduction to the theory of rough paths, we also refer the reader to the monograph \cite{FH14}.

Our ultimate aim is to handle the singularity in \eqref{eq:u} caused by the term of order $\varepsilon^{-1/2}$.
 Specifically, this term is expected to converge to a stochastic integral, and we aim to quantify this convergence within the framework of rough paths. To this end, we introduce the process 
  $y^{\varepsilon}$, defined by
\begin{equation}
  \frac{d y^{\varepsilon}}{ d t} = \varepsilon^{- 1 / 2} w^{\varepsilon}, \qquad
  y^{\varepsilon}_0 = 0, \label{eq:y}
\end{equation}
so that the critical term in \eqref{eq:u} can be expressed as
$
b(\varepsilon^{- 1 / 2} w^{\varepsilon},u^{\varepsilon})=b\left({d y^{\varepsilon}}/{ d t},u^{\varepsilon}\right).
$
This is where rough path theory becomes crucial: although the time derivatives ${d y^{\varepsilon}}/{ d t}$	
  exhibit singular behavior and blow up as $\varepsilon \to 0$, for instance in finite variation, it is still possible to establish their convergence in a suitably weaker sense. This convergence is nonetheless strong enough to allow us to pass to the limit in the entire system \eqref{eq:u}, \eqref{eq:w}, \eqref{eq:r}.

To this end, we define $(Y^{\varepsilon, 1},Y^{\varepsilon, 2})$ as the canonical rough path lift of $y^{\varepsilon}$ corresponding to the first- and second-order  iterated integrals:
\[ Y^{\varepsilon, 1}_{s t} \assign  
    \int_s^{t} \varepsilon^{- 1 / 2} w_r^{\varepsilon} d r
   , \qquad Y^{\varepsilon, 2}_{s t} \assign \int_s^t (
   y^{\varepsilon}_{r}-y^{\varepsilon}_{s}) \otimes d y_r^{\varepsilon}, \qquad 0 \leqslant s< t
   \leqslant T. \]
This object is well-defined since $y^{\varepsilon}$ exhibits sufficient regularity in time. Our goal is to establish that, in an appropriate rough path framework, 
$(Y^{\varepsilon, 1},Y^{\varepsilon, 2})$ converges to a well-defined limit.
More specifically, the convergence occurs in the space of  $\alpha$-H\"older continuous rough paths, with values in a suitable Sobolev space, for every $\alpha\in (0,1/2)$. Implicitly, the first component $Y^{\varepsilon, 1}$ is $\alpha$-H\"older continuous in time, while the second component $Y^{\varepsilon, 2}$ enjoys $2\alpha$-H\"older continuity in time, uniformly with respect to $\varepsilon$. 
Furthermore, we  identify the limit rough path, which, rather surprisingly, deviates from the classical Stratonovich lift of a Wiener process.

Our first main result can be stated as  follows, with the precise formulation -- including the function spaces -- presented  in Theorem~\ref{l:11}. The detailed proof can be found in  Section~\ref{s:2}.

\begin{theorem}\label{thm:main1}
\begin{enumerate}
\item The canonical rough path lift $(Y^{\varepsilon, 1},Y^{\varepsilon, 2})$ converges as $\varepsilon\to 0$ a.s.  in the sense of rough paths to a
 rough path lift $(B^1, B^2)$ of the Wiener process $B\assign (- C)^{- 1} Q^{1 / 2} W$. The second component is given in terms of
  It{\^o}'s stochastic integration as
  \begin{equation*}
    B^2_{s t} \assign \int_s^t \delta B_{s r} \otimes d B_r + (t - s) \int w
    \otimes (- C)^{- 1} w d \mu (w), 
  \end{equation*}
  with $\mu$ being the unique invariant measure of \eqref{eq:w}. Alternatively, $B^2$ is given in terms of Stratonovich's
  stochastic integration as
  \begin{equation*}
    B^2_{s t} = \int_s^t \delta B_{s r} \otimes \circ d B_r + (t - s) M,
  \end{equation*}
  where $M$ is antisymmetric and explicitly defined in terms of operators $C$ and $Q$.
  \item If the operators $C$ and $Q$ commute and $C$ is symmetric
then $M=0$ and $(B^1, B^2)$  is the Stratonovich lift of $B$.
  \end{enumerate}

\end{theorem}

We observe  that the rough path lift of $B$ obtained in the limit generally differs from the Stratonovich lift via
an antisymmetric finite variation part, denoted by $M$. Nevertheless, due to the antisymmetry of $M$, the lift $(B^{1},B^{2})$ remains a geometric rough path, as expected for limits of canonical lifts of smooth paths. 
Under the assumptions of point \emph{2.}, we  recover the result of \cite{DP22}, albeit via  a completely different approach. Moreover, if $Q$ is the identity on $\mathbb{R}^{d}$ and $-C$ is a $d\times d$ matrix with strictly positive real parts of all eigenvalues, we not only  recover but also extend the result of Theorem 3.8 in \cite{FH14}. This earlier work examined the behavior of  the zero-mass  limit of  physical Brownian motion  in a magnetic field, where a nontrivial  antisymmetric part $M$ (in our notation) also emerged in the  limit rough path. However, in that setting, the
  non-triviality of $M$ stemmed from the non-symmetry of
  $C$. In contrast, our result demonstrates that $M$ can still be nontrivial even when 
  $C$ is symmetric, provided $C$ does not commute with $Q$.

The proof of Theorem~\ref{thm:main1} is based on the Kolmogorov continuity criterion for rough paths, specifically Theorem~3.3 in \cite{FH14}. In particular, the convergence of  the second component of the
rough path is derived from  ergodic properties of the Ornstein--Uhlenbeck process 
through a suitable variant of an ergodic theorem, which we develop  in
Section~\ref{s:erg}.

\medskip

The result of Theorem~\ref{thm:main1} serves as the first essential  building block for passing  to the limit in the system \eqref{eq:u}, \eqref{eq:w}, \eqref{eq:r}. The second crucial aspect is identifying a suitable  notion of solution to this system that remains  stable under the convergence provided by Theorem~\ref{thm:main1}. This notion must also be compatible with the uniform   estimates on $u^{\varepsilon}, w^{\varepsilon}, r^{\varepsilon}$ as $\varepsilon\to 0$, ensuring that it is possible   to pass to the limit as $\varepsilon\to 0$ and identify the limit equation.

At this stage, we prefer not to discuss the precise definition of such a solution, as it requires a more sophisticated framework involving the theory of unbounded rough drivers,  which can be understood as  unbounded operator-valued rough paths. For further background on this formalism, we refer to \cite{BG17}, where unbounded rough drivers were originally introduced. Details related to our specific setting can be found  in Section~\ref{s:RP}, and the derivation of the rough path formulation of \eqref{eq:u}, \eqref{eq:w}, \eqref{eq:r} is discussed in Section~\ref{s:rr}. Specifically,  in Definition~\ref{d:sol}, we introduce the concept of a probabilistically weak rough path solution to \eqref{eq:u}, \eqref{eq:w}, \eqref{eq:r}. This  is a solution that is weak in both analytical and probabilistic sense, and it satisfies the equations in the sense of rough paths.

 Finally, we arrive to our second main result, which is formulated  in Theorem~\ref{thm:main}. The proof of this result is presented in  Section~\ref{s:IS} and Section~\ref{s:tight1}.

\begin{theorem}
  \label{thm:main2}
  Let the initial values $(u_0^{\varepsilon})_{\varepsilon \in
  (0, 1)}$ and $(v_0^{\varepsilon})_{\varepsilon \in (0, 1)}$ be given
  so that both $(u_0^{\varepsilon})_{\varepsilon \in (0, 1)}$
  and $(\varepsilon^{1 / 2} v^{\varepsilon}_0)_{\varepsilon \in (0, 1)}$ are
  bounded in $L^{2}$ uniformly in $\varepsilon$. There exist
 probabilistically weak rough path solutions $(u^{\varepsilon}, (- C)^{- 1} Q^{1 / 2} W^{\varepsilon})$\footnote{{Recall that, for probabilistically weak solutions, the Brownian motion in \eqref{eq:w} is part of the solution and hence it a priori depends on $\varepsilon$, leading to the replacement of $W$ in \eqref{eq:w} by the Brownian motion $W^{\varepsilon}$.}} to \eqref{eq:u},
  \eqref{eq:w}, \eqref{eq:r}
  that   converge in
  law to a probabilistically weak rough path solution to 
  \begin{equation}
  d u = A u d t + b (u, u) d t + b \left( \overline{r}, u \right) d t + b
  (\ast d B, u) , \label{eq:ulim}
\end{equation}
where $\overline{r}$ is the so-called It\^o--Stokes velocity
\begin{equation}
  \overline{r} \assign \int (- C)^{- 1} b (w, w) d \mu (w)
  \label{eq:itostokes}
\end{equation}
with $\mu$ being the unique invariant measure of \eqref{eq:w},
and the stochastic integral with respect to $B$ corresponds to the
integration with respect to the rough path $(B^1, B^2)$ obtained in Theorem~\ref{thm:main1}.
\end{theorem}

As demonstrated, the rough path theory serves as a valuable tool for directly identifying the limit stochastic integral. This approach may hold potential interest for other problems in the context of approximation--diffusion.

Overall, our construction is implicitly probabilistic,  and the
convergence of the system \eqref{eq:u}, \eqref{eq:w}, \eqref{eq:r} to \eqref{eq:ulim}
is established in law. On the one hand, the convergence of $r^{\varepsilon}$ to the It\^o--Stokes velocity, as proved in Section~\ref{s:av}, leverages the ergodic properties of the Ornstein--Uhlenbeck process discussed in Section~\ref{s:erg}.
On the other hand,  the final passage necessitates a combination of  probabilistic and
rough path arguments. To this end,   we develop a rough path variant of the stochastic compactness argument, which is  based on the Skorokhod
representation theorem. A purely rough path approach is hindered among others
by the fact that the a priori uniform estimate for $r^{\varepsilon}$ 
 holds only in expectation (cf. Section~\ref{s:bdr}). Consequently,  the analysis of the rough path
formulation of \eqref{eq:u}, particularly regarding the remainder estimates in
Section~\ref{s:rem} and the time regularity of $u^{\varepsilon}$ in
Section~\ref{s:timereg}, becomes uniform only after taking expectation. This issue complicates  the passage to the limit in Section~\ref{s:tight} and requires careful consideration.

\subsection*{Acknowledgment} The research of M.H. was funded by the European Research Council
  (ERC) under the European Union's Horizon 2020 research and innovation
  programme (grant agreement No. 949981). A.D. benefits from the support of the French government “Investissements d’Avenir” program integrated to France 2030, bearing the following reference ANR-11-LABX-0020-01 and is partially funded by the ANR project ADA.

\section{Preliminaries}\label{s:2222}

\subsection{Function spaces, operators and noise}\label{s:21}

\rmb{We first introduce the basic functional setting related to the Navier--Stokes equations. The reader is referred to \cite[Chapter 2]{T83} for further details.} Let $H$ denote the $L^2 (\mathbb{T}^3 ; \mathbb{R}^3)$ space of vector fields
of zero mean and divergence and let $\langle \cdummy, \cdummy \rangle$ be the
associated inner product. By $(e_k)_{k \in \mathbb{N}}$ we denote an orthonormal basis of $H$. $\mathbb{P}$ denotes the Leray projection onto $H$ \rmb{which is defined by $\mathbb{P}u=u-\nabla\Delta^{-1}\div u$.}
$A$ is the Stokes operator $A u = \nu \Delta \mathbb{P}$ for a viscosity
constant $\nu > 0$. \rmb{Here $\Delta$ is the Laplace operator on $\mathbb T^3$, {\it i.e.} with periodic boundary conditions.}
By
${H^n} $, $n \in \mathbb{R}$, we denote the domains of fractional powers of $-
A$ equipped with the graph norm
\[ \| u \|_{H^n} \assign \| (- A)^{n / 2} u \|_H . \]
$b$ is the Navier--Stokes nonlinearity
\[ b (u, v) = -\mathbb{P} (u \cdummy \nabla v) . \]
It is a continuous operator $H \times H \rightarrow H^{- \theta_0}$
as well as $H \times H^1 \rightarrow H^{- \theta_0 + 1}$, where
$\theta_0 > 1 + 3 / 2$, \rmb{as can be seen from its bilinearity combined with the estimates
\[ | \langle b (u, v), w \rangle | \leqslant \| u \|_H \| v \|_{H} \| \nabla w
   \|_{L^{\infty}}\lesssim \| u \|_H \| v \|_{H} \|  w
   \|_{H^{\theta_{0}}},\]
   \[ | \langle b (u, v), w \rangle | \leqslant \| u \|_H \| \nabla v \|_{H} \| w
   \|_{L^{\infty}}\lesssim  \| u \|_H \| v \|_{H^1} \| w
   \|_{H^{\theta_{0}-1}},\]
which hold true for smooth divergence free vector fields $u, v, w$ by the Sobolev embedding.}
We fix a $\theta_0$ close to $1 + 3 / 2$ throughout
the paper.
Moreover, the symmetry property $\langle b (u, v), w \rangle = - \langle b (u,
w), v \rangle$ leads to the cancellation $\langle b (u, v), v \rangle = 0$.

\rmb{Next, we are concerned with assumptions on the operator $C$ and we refer the reader to \cite{EN06} for an introduction to semigroup theory.} Let $\sigma > 2 + 3 / 2$ be fixed throughout the paper, close to the given
lower bound. The operator $C$ is the infinitesimal generator of a strongly
continuous semigroup $(e^{C t})_{t \geqslant 0}$ on
${H^{\sigma}}$, $- C$ is invertible and the adjoint $C^{\ast}$
generates the adjoint semigroup $(e^{C^{\ast} t})_{t \geqslant 0} = ((e^{C
t})^{\ast})_{t \geqslant 0}$. We assume that there exist $\iota > 0$, $\gamma
\geqslant 0$ \ such that 
\begin{equation}
  \| e^{C t} \|_{\mathcal{L} (H^{\sigma})} \lesssim e^{- \iota t},
  \label{eq:6}
\end{equation}
\begin{equation}
  \| (- C)^{- 1} (e^{C (t - s)} - \tmop{Id}) \|_{\mathcal{L} (H^{\sigma} ;
  H^{\sigma})} \lesssim | t - s |^{1 / 2}, \label{eq:77}
\end{equation}
\begin{equation}
  \| (- C)^{- 1} \|_{\mathcal{L} (H^{\sigma} ; H^{\sigma})} + \| (- C)^{- 1}
  \|_{\mathcal{L} (H^{- \theta_0} ; H^{- \theta_0})} \lesssim 1, \label{eq:89}
\end{equation}
\begin{equation}
  - \langle w, C w \rangle \gtrsim \| w \|_{H^{\gamma}}^{{2}} . \label{eq:999}
\end{equation}

An example of such an operator is given by $C = - \rho (- A)^{\varsigma} + K$
for some $\rho > 0$, $\varsigma \geqslant 0$ and a suitable (lower order)
perturbation $K$.  \rmb{For $\gamma = 0$, we can take $\varsigma =0$ and $K$ a bounded operator on $H$ of norm less than $\rho$.  For $\gamma>0$, we can take $\varsigma\ge \gamma$  and $K$ a bounded operator from $H^{\tilde\gamma+\sigma}$ to $H^\sigma$ with $\tilde\gamma<\gamma$ so that $C$ is a sectorial orperator in $H^\sigma$ and \eqref{eq:77} holds (see \cite{Henry}). Under the additional assumption that $K(-A)^{-\varsigma}$ is bounded from $H^\sigma$ to itself and from $H^{-\theta_0}$ to itself with sufficiently small norm, the other assumptions hold by perturbation.  }

\rmb{We proceed with the assumptions on the noise. For foundations of infinite-dimensional stochastic analysis we refer to \cite{DPZa14}.}
The driving process $W$ is a cylindrical Wiener process on $H$ and the
covariance operator $Q$ is symmetric, nonnegative and trace class and
\begin{equation}
  \| Q^{1 / 2} \|_{L_2 (H ; H^{\sigma})} \lesssim 1, \label{eq:tr}
\end{equation}
where $L_2 (H ; H^{\sigma})$ is the space of Hilbert--Schmidt operators from $H$ to $H^\sigma$. 
Additionally, we make the slightly stronger assumption that there exists $\vartheta \in (0, 1]$ so that
\begin{equation}
  \| (- C)^{\vartheta} Q^{1 / 2} \|_{L_2 (H ; H^{\sigma})} \lesssim 1.
  \label{eq:new}
\end{equation}
Finally, we have:
\begin{equation}
  \| (- C)^{- \vartheta} (e^{C (t - s)} - \tmop{Id}) \|_{\mathcal{L} (H^{\sigma}
  ; H^{\sigma})} \lesssim | t - s |^{\vartheta / 2} . \label{eq:new2}
\end{equation}
Indeed, this is clearly true for $\vartheta=0$ so that this follows by interpolation with \eqref{eq:77}.

\subsection{Rough path theory}\label{s:RP}

\rmb{In this section, we introduce the formalism of the rough path theory and particularly the concept of unbounded rough driver introduced in \cite{BG17}. For a thorough introduction to the theory of rough paths,  we refer to \cite{FH14}.}

Throughout the paper, $T > 0$ is given. For a path $g : [0, T] \rightarrow E$
to a Banach space $E$ we denote its increment by $\delta g_{s t} = g_t - g_s$,
$0 \leqslant s \leqslant t \leqslant T$. Let $\Delta_T \assign \{ (s, t) \in
[0, T]^2 ; s \leqslant t \}$. For a two-index map $g : \Delta_T \rightarrow E$
we define its increment by $\delta g_{s r t} = g_{s t} - g_{s r} - g_{r t}$.

For a Hilbert space $E$, $E \otimes E$ denotes the completion of the algebraic
tensor $E \otimes_{\tmop{alg}} E$ with respect to the Hilbert--Schmidt norm.
More precisely, we define the inner product on $E \otimes_{\tmop{alg}} E$ by
\[ \langle a \otimes_{\tmop{alg}} b, c \otimes_{\tmop{alg}} d \rangle_{E
   \otimes E} = \langle a, c \rangle_E \langle b, d \rangle_E \]
and $E \otimes E$ is then the completion of $E \otimes_{\tmop{alg}} E$ with
respect to the norm $\|\cdummy\|_{E \otimes E} = \langle \cdummy, \cdummy
\rangle_{E \otimes E}^{1 / 2}$. This space can be idendified with the space of
Hilbert--Schmidt operators $L_2 (E) \assign L_2 (E ; E)$, see \cite{RS80}.

A two-index map $\omega : \Delta_T \rightarrow [0, \infty)$ is called control
if it is continuous, superadditive, i.e. for all $s \leqslant r \leqslant t$
\[ \omega (s, r) + \omega (r, t) \leqslant \omega (s, t), \]
and $\omega (s, s) = 0$ for all $s \in [0, T]$.

Let $\alpha \in (0, 1]$. We denote by $C^{\alpha}_2 ([0, T] ; E)$ the closure
of the set of smooth two-index maps $g : \Delta_T \rightarrow E$ with respect
to the seminorm
\[ \| g \|_{C^{\alpha}_2 ([0, T] ; E)} \assign \sup_{s, t \in \Delta_T, s \neq
   t} \frac{\| g_{s t} \|_E}{| t - s |^{\alpha}} . \]
\rmb{Taking closure of smooth functions guarantees separability and, consequently,} the space $C^{\alpha}_2 ([0, T] ; E)$ is
Polish.

Let $p \geqslant 1$. By $C^{p - \tmop{var}}_2 ([0, T] ; E)$ we denote the
space of continuous two-index maps $g : \Delta_T \rightarrow E$ with finite
$p$-variation, i.e.
\[ \| g \|_{C^{p - \tmop{var}}_2 ([0, T] ; E)} \assign \sup_{\pi \in
   \mathcal{P} ([0, T])} \left( \sum_{(s, t) \in \pi} \| g_{s t} \|_E^p
   \right)^{1 / p} < \infty, \]
where the supremum is taken over all partitions $\pi$ of $[0, T]$.

By $C^{p - \tmop{var}}_{2, \tmop{loc}} ([0, T] ; E)$ we denote the space of
two-index maps $g : \Delta_T \rightarrow E$ such that there exists a covering
$\{ I_k \}_k$ of $[0, T]$ so that $g \in C^{p - \tmop{var}}_2 (I_k ; E)$ \rmb{for all $k$}. We
denote by $C^{p - \tmop{var}} ([0, T] ; E)$ the set of all paths $g : [0, T]
\rightarrow E$ so that $\delta g \in C^{p - \tmop{var}}_2 ([0, T] ; E)$. We
note that $C^{\alpha}_2 ([0, T] ; E) \subset C^{p - \tmop{var}}_2 ([0, T] ;
E)$ provided $p = 1 / \alpha$. Also, if $g \in C^{p - \tmop{var}}_2 ([0, T] ;
E)$ then
\[ \omega_q (s, t) \assign \| g \|_{C^{p - \tmop{var}}_2 ([\rmb{s,t}] ; E)}^p \]
is a control.

We proceed with a definition of a Hilbert space-valued rough path.

\begin{definition}
  \label{d:rp}Let $\alpha \in (1 / 3, 1 / 2)$. A continuous $E$-valued
  $\alpha$-rough path is a pair
  \[ Y = (Y^1, Y^2) \in C^{\alpha}_2 ([0, T] ; E) \times C^{2 \alpha}_2 ([0,
     T] ; E \otimes E) \]
  which satisfies Chen's relation
  \[ \delta Y^1_{s r t} = 0, \qquad \delta Y^2_{s r t} = Y^1_{s r} \otimes
     Y^1_{r t}, \qquad 0 \leqslant s \leqslant r \leqslant t \leqslant T. \]
\end{definition}

The space of continuous $E$-valued $\alpha$-rough paths is denoted by
$\mathcal{C}^{\alpha} ([0, T] ; E)$.

For the trace class covariance
operator $Q$ as defined above, the $Q$-Wiener process $Q^{1 / 2} W$ is given
by the formula $Q^{1 / 2} W = \sum_{k \in \mathbb{N}} Q^{1 / 2} e_k W_k$ for a
sequence $(W_k)_{k \in \mathbb{N}}$ of mutually independent real-valued Wiener
processes. Its Stratonovich rough path lift reads as
\[ W^1_{s t} \assign (\delta Q^{1 / 2} W)_{s t} = \sum_{k \in \mathbb{N}} Q^{1
   / 2} e_k (\delta W_k)_{s t}, \]
\[ W^2_{s t} \assign \int_s^t (\delta Q^{1 / 2} W)_{s r} \otimes \circ d Q^{1
   / 2} W_r = \sum_{k, \ell \in \mathbb{N}} \int_s^t (\delta W_k)_{s r} \circ
   d W_{\ell, r} (Q^{1 / 2} e_k \otimes Q^{1 / 2} e_{\ell}) \]
and is a continuous {$H^{\sigma}$-valued} $\alpha$-rough path,
see e.g. Exercise 3.16 in {\cite{FH14}}. In particular, the regularity of the
second component follows via Kolmogorov's continuity criterion (see e.g.
Theorem~3.1 in {\cite{FH14}}) together with Nelson's estimate for $p \in [2,
\infty)$ and It{\^o}'s isometry: rewriting the stochastic integral in
It{\^o}'s form gives
\[ W^2_{s t} = \int_s^t (\delta Q^{1 / 2} W)_{s r} \otimes d Q^{1 / 2} W_r +
   \frac{1}{2} \sum_{k \in \mathbb{N}} (t - s) (Q^{1 / 2} e_k \otimes Q^{1 /
   2} e_k) . \]
Here the It{\^o}--Stratonovich correction is of finite variation
{in $H^{\sigma} \otimes H^{\sigma}$ by \eqref{eq:tr}}, whereas
for the It{\^o} stochastic integral it holds by the trace class assumption
together with It{\^o}'s isometry
\[ \mathbb{E} \left[ \left\| \int_s^t (\delta Q^{1 / 2} W)_{s r} \otimes d
   Q^{1 / 2} W_r \right\|_{H^{\sigma} \otimes H^{\sigma}}^2 \right]
   =\mathbb{E} \left[ \int_s^t \| (\delta Q^{1 / 2} W)_{s r} \otimes Q^{1 / 2}
   \cdummy \|_{L_2 (H ; H^{\sigma} \otimes H^{\sigma})}^2 d r \right], \]
where
\[ \| (\delta Q^{1 / 2} W)_{s r} \otimes Q^{1 / 2} \cdummy \|_{L_2 (H ;
   H^{\sigma} \otimes H^{\sigma})}^2 = \sum_{\ell \in \mathbb{N}} \| (\delta
   Q^{1 / 2} W)_{s r} \otimes Q^{1 / 2} e_{\ell} \|_{H^{\sigma} \otimes
   H^{\sigma}}^2 \]
\[ = \| (\delta Q^{1 / 2} W)_{s r} \|^2_{H^{\sigma}} \sum_{\ell \in
   \mathbb{N}} \| Q^{1 / 2} e_{\ell} \|_{H^{\sigma}}^2 = \| Q^{1 / 2} \|_{L_2
   (H ; H^{\sigma})}^4 | (\delta W_k)_{s r} |^2, \]
and therefore by \eqref{eq:tr}
\[ \mathbb{E} \left[ \left\| \int_s^t (\delta Q^{1 / 2} W)_{s r} \otimes d
   Q^{1 / 2} W_r \right\|_{H^{\sigma} \otimes H^{\sigma}}^2 \right] \lesssim
   \int_s^t \mathbb{E} | (\delta W_k)_{s r} |^2 d r \lesssim \int_s^t | r - s
   | d r \lesssim | t - s |^2 . \]
Accordingly,
\[ (\mathbb{E} [\| W^2_{s t} \|_{H^{\sigma} \otimes H^{\sigma}}^p])^{1 / p}
   \lesssim (\mathbb{E} [\| W^2_{s t} \|_{H^{\sigma} \otimes
   H^{\sigma}}^2])^{1 / 2} \lesssim | t - s | . \]
Unbounded rough drivers can be seen as operator valued rough paths with values
in the space of unbounded operators. In what follows, we work with a scale of
Banach spaces $(E^m, \| \cdummy \|_m)_{m \in \mathbb{R}_+}$ such that $E^{m +
l}$ is continuously embedded in $E^m$ for every $m, l \in \mathbb{R}_+$. We
denote by $E^{- m}$ the topological dual of $E^m$.

\begin{definition}
  \label{d:urd}Let $\alpha \in (1 / 3, 1 / 2)$. A continuous unbounded
  $\alpha$-rough driver with respect to the scale $(E^m, \| \cdummy \|_m)_{m
  \in \mathbb{R}_+}$ is a pair $\mathbb{A}= (\mathbb{A}^1, \mathbb{A}^2)$ of
  two-index maps such that
  \begin{equation}
    \| \mathbb{A}^1_{s t} \|_{\mathcal{L} (E^{- m}, E^{- (m + 1)})} \lesssim |
    t - s |^{\alpha}, \qquad m \in [0, 2], \label{eq:a1a1}
  \end{equation}
  \begin{equation}
    \| \mathbb{A}^2_{s t} \|_{\mathcal{L} (E^{- m}, E^{- (m + 2)})} \lesssim |
    t - s |^{2 \alpha}, \qquad m \in [0, 1], \label{eq:a2a2}
  \end{equation}
  and Chen's relation holds true, i.e.
  \begin{equation}
    \delta \mathbb{A}^1_{s r t} = 0, \qquad \delta \mathbb{A}^2_{s r t}
    =\mathbb{A}^1_{r t} \mathbb{A}^1_{s r}, \qquad 0 \leqslant s \leqslant r
    \leqslant t \leqslant T. \label{eq:chenA}
  \end{equation}
\end{definition}

The above definition appeared already in a number of works, see e.g.
{\cite{DGHT19}}, {\cite{FHLN20}}, {\cite{MR3918521}}, {\cite{HLN21}}. In our
present setting we let $E^m = H^m$, $m \in \mathbb{R}$. In order to control
the term leading to the It{\^o}--Stokes drift in Section~\ref{s:rem} below, we
will make use of the bounds
\begin{equation}
  \| \mathbb{A}^{1, \ast}_{s t} \varphi \|_{L^{\infty}} \lesssim \| \varphi
  \|_{H^3} | t - s |^{\alpha}, \label{eq:a1a11}
\end{equation}
\begin{equation}
  \| \mathbb{A}^{2, \ast}_{s t} \varphi \|_{L^{\infty}} \lesssim \| \varphi
  \|_{H^{\theta_0 + 1}} | t - s |^{2 \alpha} . \label{eq:a2a22}
\end{equation}
By Sobolev imbedding \eqref{eq:a1a11} follows from \eqref{eq:a1a1} with $m \in
(3 / 2, 2]$ but \eqref{eq:a2a22} does not follow from \eqref{eq:a2a2} and
needs to be verified. We also note that it would suffice if \eqref{eq:a1a11},
\eqref{eq:a2a22} were true with arbitrary positive exponents on their right
hand sides, not necessarily $\alpha$ and $2 \alpha$, respectively.

\section{Rough formulation of the problem}\label{s:rr}

We intend to use rough path theory to formulate the equation \eqref{eq:u} for
$u^{\varepsilon}$, to obtain uniform estimates and to pass to the limit. In
order to apply the theory of {\cite{DGHT19}} and {\cite{MR3918521}}, we
understand the noise term $\varepsilon^{- 1 / 2} b (w^{\varepsilon},
u^{\varepsilon}) d t$ as an unbounded rough driver term of the form
$d\mathbb{A}^{\varepsilon} u^{\varepsilon}$. More precisely, we define
$y^{\varepsilon}$ via \eqref{eq:y} so that we obtain formally
\begin{equation}
  \varepsilon^{- 1 / 2} b (w^{\varepsilon}, u^{\varepsilon}) d t = -\mathbb{P}
  [d y^{\varepsilon} \cdummy \nabla u^{\varepsilon}], \label{eq:2}
\end{equation}
where $\mathbb{P}$ denotes the Leray projection. We show in Theorem~\ref{l:11}
below, that the so-defined process $y^{\varepsilon}$, lifted canonically to a
rough path, converges to a certain rough path lift of the Brownian motion $B =
(- C)^{- 1} Q^{1 / 2} W$ in the sense of rough paths $\mathcal{C}^{\alpha}
([0, T] ; H^{\sigma})$. In other words, letting
\[ Y^{\varepsilon, 1}_{s t} \assign \delta y_{s t}^{\varepsilon} = \delta
   \left( \int_0^{\point} \varepsilon^{- 1 / 2} w_r^{\varepsilon} d r
   \right)_{s t}, \qquad Y^{\varepsilon, 2}_{s t} \assign \int_s^t \delta
   y^{\varepsilon}_{s r} \otimes d y_r^{\varepsilon}, \qquad 0 \leqslant s, t
   \leqslant T, \]
defines a rough path in the sense of Definition~\ref{d:rp}. In particular,
under our assumptions on $C$, $Q$, it is an $\mathcal{C}^{\alpha} ([0, T] ;
H^{\sigma})$-valued rough path.

As the next step, we shall compose it additionally with the gradient and the
Leray projection as required in \eqref{eq:2}, in order to define the
corresponding unbounded rough driver $\mathbb{A}^{\varepsilon}$. In order to
guess the form for its second component, we first work formally and iterate
the equation \eqref{eq:u} into itself. Focusing only on the key terms
containing $w^{\varepsilon}$, we obtain
\[ \delta u_{s t}^{\varepsilon} = \cdots - \int_s^t \mathbb{P} [\varepsilon^{-
   1 / 2} w^{\varepsilon}_r \cdummy \nabla u^{\varepsilon}_r] d r \]
\[ = \cdots - \int_s^t \mathbb{P} [\varepsilon^{- 1 / 2} w^{\varepsilon}_r
   \cdummy \nabla u^{\varepsilon}_s] d r + \int_s^t \int_s^r \mathbb{P}
   [\varepsilon^{- 1 / 2} w^{\varepsilon}_r \cdummy \nabla \mathbb{P}
   [\varepsilon^{- 1 / 2} w^{\varepsilon}_{\theta} \cdummy \nabla
   u^{\varepsilon}_{\theta}]] d \theta d r + \cdots \]
\begin{equation}
  = \cdots - \int_s^t \mathbb{P} [\varepsilon^{- 1 / 2} w^{\varepsilon}_r
  \cdummy \nabla u^{\varepsilon}_s] d r + \int_s^t \int_s^r \mathbb{P}
  [\varepsilon^{- 1 / 2} w^{\varepsilon}_r \cdummy \nabla \mathbb{P}
  [\varepsilon^{- 1 / 2} w^{\varepsilon}_{\theta} \cdummy \nabla
  u^{\varepsilon}_s]] d \theta d r + u^{\varepsilon, \natural}_{s t},
  \label{eq:A22}
\end{equation}
where we expect the remainder $u^{\varepsilon, \natural}_{s t}$ to be of order
$o (| t - s |)$ as a distribution with respect to the spatial variable. In the
first term on the right hand side of \eqref{eq:A22} we recognize the first
component of the unbounded rough driver and we define
\begin{equation}
  \qquad \mathbb{A}^{\varepsilon, 1}_{s t} \varphi \assign -\mathbb{P}
  [Y^{\varepsilon, 1}_{s t} \cdummy \nabla \varphi] . \label{eq:A1}
\end{equation}

To identify the second component of the rough driver from the second term on
the right hand side of \eqref{eq:A22}, we first recall that by definition
\[ Y^{\varepsilon, 2}_{s t} = \int_s^t \delta \left( \int_0^. \varepsilon^{- 1
   / 2} w_{\theta}^{\varepsilon} d \theta \right)_{s r} \otimes d \left(
   \int_0^. \varepsilon^{- 1 / 2} w_{\theta}^{\varepsilon} d \theta \right)_r
   = \int_s^t \int_s^r \varepsilon^{- 1 / 2} w_{\theta}^{\varepsilon} d \theta
   \otimes \varepsilon^{- 1 / 2} w^{\varepsilon}_r d r. \]
Consequently, we define for $\varphi \in C^{\infty} (\mathbb{T}^3)$
\begin{equation}
  (\mathbb{A}^{\varepsilon, 2}_{s t} \varphi) (x) \assign A^2 (Y^{\varepsilon,
  2}_{s t}, \varphi) (x), \label{eq:A2}
\end{equation}
where the operator $A^2$ does not depend on $\varepsilon$, is bilinear and
acts on $(H^{\sigma} \otimes H^{\sigma}) \times C^{\infty} (\mathbb{T}^3)$ as
\begin{equation}
  A^2 (f \otimes g, \varphi) (x) =\mathbb{P} [g (x) \cdummy \nabla \mathbb{P}
  [f (x) \nosymbol \cdummy \nabla \varphi (x)]] . \label{eq:2a2}
\end{equation}
With this definition, the second term on the right hand side of \eqref{eq:A22}
equals to $\mathbb{A}^{\varepsilon, 2}_{s t} u^{\varepsilon}_s$.

In addition, using the Chen's relation for rough paths (which in particular
holds true for canonical lifts given by Lebesgue--Stieltjes integration),
namely,
\[ \delta Y^{\varepsilon, 2}_{s r t} = Y^{\varepsilon, 1}_{s r} \otimes
   Y^{\varepsilon, 1}_{r t}, \]
together with the bilinearity of the operator $A^2$, we obtain
\[ \delta (\mathbb{A}^{\varepsilon, 2} \varphi)_{s r t} = (\delta
   \mathbb{A}^{\varepsilon, 2})_{s r t} \varphi = A^2 (\delta Y^{\varepsilon,
   2}_{s r t}, \varphi) \]
\[ = A^2 (Y^{\varepsilon, 1}_{s r} \otimes Y^{\varepsilon, 1}_{r t}, \varphi)
   =\mathbb{P} [Y^{\varepsilon, 1}_{r t} \cdummy \nabla \mathbb{P}
   [Y^{\varepsilon, 1}_{s r} \cdummy \nabla \varphi]]
   =\mathbb{A}^{\varepsilon, 1}_{r t} \mathbb{A}^{\varepsilon, 1}_{s r}
   \varphi . \]
Hence $(\mathbb{A}^{\varepsilon, 1}, \mathbb{A}^{\varepsilon, 2})$ satisfies
the Chen's relation \eqref{eq:chenA}.

It remains to verify the corresponding analytic estimates \eqref{eq:a1a1},
\eqref{eq:a2a2} as well as \eqref{eq:a2a22}. For our application we work with
the $L^2$-scale of Sobolev spaces defined in Section~\ref{s:21}, namely, we
let $E^m \assign H^m$, $m \in \mathbb{R}$. From the continuity of the Leray
projection on $H^m$, $m \in \mathbb{R}$, together with the fact that
$\tmop{div} Y^{\varepsilon, 1}_{s t} = 0$ we obtain for $m \in [0, 2]$ and
$\varphi \in H^{- m}$, $\| \varphi \|_{H^{- m}} \leqslant 1$
\[ \| \mathbb{A}_{s t}^{\varepsilon, 1} \varphi \|_{H^{- (m + 1)}} =
   \sup_{\psi \in H^{m + 1}, \| \psi \|_{H^{m + 1}} \leqslant 1} \langle
   Y^{\varepsilon, 1}_{s t} \cdummy \nabla \varphi, \psi \rangle = \sup_{\psi
   \in H^{m + 1}, \| \psi \|_{H^{m + 1}} \leqslant 1} \langle Y^{\varepsilon,
   1}_{s t} \cdummy \nabla \psi, \varphi \rangle \]
\[ \leqslant \sup_{\psi \in H^{m + 1}, \| \psi \|_{H^{m + 1}} \leqslant 1} \|
   Y^{\varepsilon, 1}_{s t} \cdummy \nabla \psi \|_{H^m} \leqslant \|
   Y^{\varepsilon, 1}_{s t} \|_{W^{m, \infty}} \lesssim \| Y^{\varepsilon,
   1}_{s t} \|_{H^{\sigma}}, \]
provided $H^{\sigma} \subset W^{m, \infty}$.

Similarly, we use $\tmop{div} y^{\varepsilon}_s = \tmop{div} \psi = 0$ and
integration by parts to obtain
\[ \langle \mathbb{A}^{\varepsilon, 2}_{s t} \varphi, \psi \rangle =
   \left\langle \int_s^t d y^{\varepsilon}_r \cdummy \nabla \mathbb{P} [\delta
   y^{\varepsilon}_{s r} \cdummy \nabla \varphi], \psi \right\rangle =
   \left\langle \varphi, \int_s^t \delta y^{\varepsilon}_{s r} \cdummy \nabla
   \mathbb{P} [d y^{\varepsilon}_r \cdummy \nabla \psi] \right\rangle \]
\[ \  \]
Accordingly, it holds for $m \in [0, 1]$ and $\varphi \in H^{- m}$, $\|
\varphi \|_{H^{- m}} \leqslant 1$
\[ \| \mathbb{A}^{\varepsilon, 2}_{s t} \varphi \|_{H^{- (m + 2)}} =
   \sup_{\psi \in H^{m + 2}, \| \psi \|_{H^{m + 2}} \leqslant 1} \left\langle
   \varphi, \int_s^t \delta y^{\varepsilon}_{s r} \cdummy \nabla \mathbb{P} [d
   y^{\varepsilon}_r \cdummy \nabla \psi] \right\rangle \]
\[ \leqslant \sup_{\psi \in H^{m + 2}, \| \psi \|_{H^{m + 2}} \leqslant 1}
   \left\| \int_s^t \delta y^{\varepsilon}_{s r} \cdummy \nabla \mathbb{P} [d
   y^{\varepsilon}_r \cdummy \nabla \psi] \right\|_{H^m} \leqslant \|
   Y^{\varepsilon, 2}_{s t} \|_{H^{\sigma} \otimes H^{\sigma}}, \]
provided $H^{\sigma} \subset W^{m + 1, \infty}$.

To summarize, for \eqref{eq:a1a1} and \eqref{eq:a2a2} we require $\sigma$ to
be such that by Sobolev embedding $H^{\sigma} \subset W^{2, \infty}$, that is,
$\sigma > 2 + 3 / 2$. Regarding \eqref{eq:a2a22}, we use the embeddings
$H^{\theta_0 + 1} \subset W^{2, \infty}$, $H^{\sigma} \subset W^{1, \infty}$
to obtain
\[ \| \mathbb{A}^{\varepsilon, 2, \ast}_{s t} \varphi \|_{L^{\infty}} =
   \left\| \int_s^t \delta y^{\varepsilon}_{s r} \cdummy \nabla \mathbb{P} [d
   y^{\varepsilon}_r \cdummy \nabla \varphi] \right\|_{L^{\infty}} \lesssim \|
   \varphi \|_{H^{\theta_0 + 1}} \| Y^{2, \varepsilon}_{s t} \|_{H^{\sigma}
   \otimes H^{\sigma}} . \]
It will be seen in Theorem~\ref{l:11} below that the desired bounds and even
convergence of $(Y^{\varepsilon, 1}, Y^{\varepsilon, 2})$ as an
$H^{\sigma}$-valued rough path in the sense of Definition~\ref{d:rp} hold
true.

Finally, in order to complete the derivation of the rough path formulation of
\eqref{eq:u}, we denote the remaining drift part in \eqref{eq:u} of finite
variation as
\begin{equation}
  \mu^{\varepsilon}_t \assign \mu^{\varepsilon, 1}_t + \mu^{\varepsilon, 2}_t
  \assign \int_0^t [A u^{\varepsilon}_s + b (u_s^{\varepsilon},
  u_s^{\varepsilon})] d s + \int_0^t b (r^{\varepsilon}_s, u_s^{\varepsilon})
  d s. \label{eq:drift}
\end{equation}
Consequently, \eqref{eq:u} rewrites as
\begin{equation}
  \delta  u^{\varepsilon}_{s t} = \delta \mu^{\varepsilon}_{s t}
  +\mathbb{A}^{\varepsilon, 1}_{s t} u^{\varepsilon}_s
  +\mathbb{A}^{\varepsilon, 2}_{s t} u^{\varepsilon}_s + u^{\varepsilon,
  \natural}_{s t}, \label{eq:urp}
\end{equation}
where $\delta f_{s t} = f_t - f_s$ and the equation is understood in the sense
of distributions with a remainder $u^{\varepsilon, \natural}_{s t}$ of order
$o (| t - s |)$. More precisely, we have the following definition.

\begin{definition}
  \label{d:sol}Let $\varepsilon \in (0, 1)$.We say that
  $((\Omega^{\varepsilon}, \mathcal{F}^{\varepsilon},
  (\mathcal{F}^{\varepsilon}_t)_{t \geqslant 0}, \tmmathbf{P}^{\varepsilon}),
  u^{\varepsilon}, r^{\varepsilon}, W^{\varepsilon})$ is a probabilistically
  weak rough path solution to \eqref{eq:u}, \eqref{eq:w}, \eqref{eq:r}
  provided
  \begin{enumerate}
    \item $(\Omega^{\varepsilon}, \mathcal{F}^{\varepsilon},
    (\mathcal{F}^{\varepsilon}_t)_{t \geqslant 0},
    \tmmathbf{P}^{\varepsilon})$ is a stochastic basis with a complete
    right-continuous filtration;
    
    \item $W^{\varepsilon}$ is a cylindrical $(\mathcal{F}^{\varepsilon}_t)_{t
    \geqslant 0}$-Wiener process on $H$;
    
    \item $r^{\varepsilon} \in L^2 (0, T ; H^{\gamma})$
    $\tmmathbf{P}^{\varepsilon}$-a.s. is $(\mathcal{F}^{\varepsilon}_t)_{t
    \geqslant 0}$-adapted and \eqref{eq:r} holds true in the analytically weak
    sense $\tmmathbf{P}^{\varepsilon}$-a.s.;
    
    \item $u^{\varepsilon} \in C_{\tmop{weak}} (0, T ; H) \cap L^2 (0, T ;
    H^1)$ $\tmmathbf{P}^{\varepsilon}$-a.s. is
    $(\mathcal{F}^{\varepsilon}_t)_{t \geqslant 0}$-adapted;
    
    \item $\mathbb{A}^{\varepsilon} = (\mathbb{A}^{\varepsilon, 1},
    \mathbb{A}^{\varepsilon, 2})$ defined via \eqref{eq:A1} and \eqref{eq:A2}
    is an unbounded $\alpha$-rough driver on $(H^m)_{m \in \mathbb{R}}$ in the
    sense of Definition~\ref{d:urd} and satisfying \eqref{eq:a2a22};
    
    \item the remainder given $\tmmathbf{P}^{\varepsilon}$-a.s. by
    \[ u^{\varepsilon, \natural}_{s t} \assign \delta  u^{\varepsilon}_{s t} -
       \delta \mu^{\varepsilon}_{s t} -\mathbb{A}^{\varepsilon, 1}_{s t}
       u^{\varepsilon}_s -\mathbb{A}^{\varepsilon, 2}_{s t} u^{\varepsilon}_s,
       \qquad 0 \leqslant s \leqslant t \leqslant T, \]
    belongs to $C^{p / 3 - \tmop{var}}_{2, \tmop{loc}} ([0, T] ; H^{- 3})$
    $\tmmathbf{P}^{\varepsilon}$-a.s. with $p = 1 / \alpha$, where the drift
    $\mu^{\varepsilon}$ was defined in \eqref{eq:drift}.
  \end{enumerate}
\end{definition}

The above definition is designed with the aim of the passage to the limit
$\varepsilon \rightarrow 0$. In particular, we did not make use of any
regularity claims which cannot be proved uniformly in $\varepsilon$.

\section{Detailed presentation of  main results}\label{s:main}

The key to our strategy lies in the following convergence result, which
consequently also implies the convergence of the unbounded rough drivers
$\mathbb{A}^{\varepsilon} = (\mathbb{A}^{\varepsilon, 1},
\mathbb{A}^{\varepsilon, 2})$ in an appropriate sense. The~proof can be found
in Section~\ref{s:2}.

\begin{theorem}
  \label{l:11}Let $y^{\varepsilon}$ be given by \eqref{eq:y}. Then the
  canonical rough path lift $(Y^{\varepsilon, 1}, Y^{\varepsilon, 2})$ of
  $y^{\varepsilon}$ given by
  \[ (Y^{\varepsilon, 1}, Y^{\varepsilon, 2}) \assign \left( \delta
     y^{\varepsilon}_{s t}, \int_s^t \delta y_{s r}^{\varepsilon} \otimes d
     y_r^{\varepsilon} \right)_{s, t \in [0, T]} \]
  converges as $\varepsilon\to 0$ to a rough path lift $(B^1, B^2)$ of the Brownian motion $B = (-
  C)^{- 1} Q^{1 / 2} W$. {The convergence holds true in $L^q (\Omega ; \mathcal{C}^{\alpha} ([0, T] ;
  H^{\sigma}))$ for all $\alpha \in (1 / 3, 1 / 2)$ and $q \in [1, \infty)$}. The second component $B^2$ is given in terms of
  It{\^o}'s stochastic integration as
  \begin{equation}
    B^2_{s t} \assign \int_s^t \delta B_{s r} \otimes d B_r + (t - s) \int w
    \otimes (- C)^{- 1} w d \mu (w), \label{eq:B2ito}
  \end{equation}
  with $\mu$ being the unique invariant measure of \eqref{eq:w}. Alternatively, $B^2$ is given in terms of Stratonovich's
  stochastic integration as
  \begin{equation}
    B^2_{s t} = \int_s^t \delta B_{s r} \otimes \circ d B_r + (t - s) M,
    \label{eq:B2str}
  \end{equation}
  where for all $k, \ell \in \mathbb{N}$
  \[ \langle M, e_k \otimes e_{\ell} \rangle = \left\langle \frac{1}{2} [(-
     C)^{- 1} Q_{\infty} - Q_{\infty} ((- C)^{- 1})^{\ast}] e_k, e_{\ell}
     \right\rangle \]
  and
  \begin{equation}
    Q_{\infty} \assign \int_0^{\infty} e^{C t} Q (e^{C t})^{\ast} d t.
    \label{eq:mu}
  \end{equation}
  In other words, $M$ is antisymmetric and $(B^1, B^2)$ is a geometric rough
  path.
  
  {Furthermore, with the convention $(Y^{0,1},Y^{0,2}):=(B^{1},B^{2})$,  the mapping
  \[ [0, 1] \rightarrow \mathcal{C}^{\alpha} ([0, T] ; H^{\sigma}), \qquad \varepsilon
     \mapsto (Y^{\varepsilon, 1},Y^{\varepsilon,2})\]
 is H{\"o}lder continuous $\tmmathbf{P}$-a.s. and for all $q\in[1,\infty)$
  \begin{equation}
    \mathbb{E} \left[\sup_{\varepsilon \in (0, 1)} \| Y^{\varepsilon, 1}
    \|_{C^{\alpha}_2 ([0, T] ; H^{\sigma})}^q\right] +\mathbb{E} \left[\sup_{\varepsilon
    \in (0, 1)} \| Y^{\varepsilon, 2} \|_{C^{2 \alpha}_2 ([0, T] ; H^{\sigma}
    \otimes H^{\sigma})}^q\right] \lesssim 1. \label{eq:new5}
  \end{equation}
 In particular, the convergence $(Y^{\varepsilon,1},Y^{\varepsilon,2})\to (B^{1},B^{2})$ as $\varepsilon\to 0$ holds true in $\mathcal{C}^{\alpha} ([0,
  T] ; H^{\sigma})$ $\tmmathbf{P}$-a.s.}
\end{theorem}

\begin{remark}
  \label{r:1}Note that $M$ is antisymmetric so the limit $(B^1, B^2)$ is
  indeed a geometric rough path, as it was expected anyway since it is a limit
  of canonical lifts of smooth paths.
  
  Moreover, if $C$ and $Q$ commute and $C$ is symmetric then $Q_{\infty} =
  \frac{1}{2} (- C)^{- 1} Q$ and $M = 0$. Hence, we recover the result of
  {\cite{DP22}} obtained by a completely different method.
  
  Finally, if $Q$ is an identity on $\mathbb{R}^d$ and $- C$ is a $d \times d$
  matrix so that all its eigenvalues have strictly positive real part, we
  recover the result of Theorem 3.8 {\cite{FH14}}. Our result is in fact stronger since we prove that the convergence holds almost surely and we have the uniform bound \eqref{eq:new5}, whereas in \cite{FH14} the supremum with respect to $\varepsilon$ is outside the expectation. In \cite{FH14}, the
  non-triviality of $M$ (in our notation) originated in the non-symmetry of
  $C$. However, our result shows that $M$ can be nontrivial even for symmetric
  $C$ provided $C$ does not commute with $Q$.
\end{remark}

Based on Theorem~\ref{l:11}, we are able to identify the limit of the rough
path formulation of \eqref{eq:u}, namely \eqref{eq:urp}, as the corresponding
rough path formulation of \eqref{eq:ulim}. As the next step, we define the
notion of probabilistically weak rough path solution to \eqref{eq:ulim}
analogously to Definition~\ref{d:sol}. To this end, we define the limit
unbounded rough driver $\mathbb{A}= (\mathbb{A}^1, \mathbb{A}^2)$ as in
\eqref{eq:A1}, \eqref{eq:A2} but using the limit rough path $(B^1, B^2)$ from
Theorem~\ref{l:11} instead of $(Y^{\varepsilon, 1}, Y^{\varepsilon, 2})$. More
precisely, we let
\begin{equation}
  \mathbb{A}^1_{s t} \varphi \assign -\mathbb{P} [B^1_{s t} \cdummy \nabla
  \varphi], \qquad \mathbb{A}^2_{s t} \varphi \assign A^2 (B^2_{s t},
  \varphi), \label{eq:ab}
\end{equation}
where $A^2$ was defined in \eqref{eq:2a2}. The limit drift is given
analogously to \eqref{eq:drift} as
\begin{equation}
  \mu_t \assign \int_0^t [A u_s + b (u_s, u_s)] d s + \int_0^t b \left(
  \overline{r}, u_s \right) d s, \label{eq:drift2}
\end{equation}
with the It{\^o}--Stokes velocity $\bar{r}$ defined in \eqref{eq:itostokes}.

\begin{definition}
  \label{d:sol2}We say that $((\Omega, \mathcal{F}, (\mathcal{F}_t)_{t
  \geqslant 0}, \tmmathbf{P}), u, B)$ is a probabilistically weak rough path
  solution to \eqref{eq:ulim} provided
  \begin{enumerate}
    \item $(\Omega, \mathcal{F}, (\mathcal{F}_t)_{t \geqslant 0},
    \tmmathbf{P})$ is a stochastic basis with a complete right-continuous
    filtration;
    
    \item $B$ is an $(\mathcal{F}_t)_{t \geqslant 0}$-Wiener process on $H$
    with covariance $(- C)^{- 1} Q ((- C)^{- 1})^{\ast}$;
    
    \item $u \in C_{\tmop{weak}} (0, T ; H) \cap L^2 (0, T ; H^1)$
    $\tmmathbf{P}$-a.s. is $(\mathcal{F}_t)_{t \geqslant 0}$-adapted;
    
    \item the remainder given by
    \[ u^{\natural}_{s t} \assign \delta  u_{s t} - \delta \mu_{s t}
       -\mathbb{A}^1_{s t} u_s -\mathbb{A}^2_{s t} u_s, \qquad 0 \leqslant s
       \leqslant t \leqslant T, \]
    belongs to $C^{p / 3 - \tmop{var}}_{2, \tmop{loc}} ([0, T] ; H^{- 3})$
    $\tmmathbf{P}$-a.s. with $p = 1 / \alpha$. Here, the drift $\mu$ was
    defined in \eqref{eq:drift2} whereas the unbounded rough driver
    $\mathbb{A}= (\mathbb{A}^1, \mathbb{A}^2)$ in \eqref{eq:ab}.
  \end{enumerate}
\end{definition}

Our main result reads as follows, the proof spreads over Section~\ref{s:2},
Section~\ref{s:IS} and Section~\ref{s:tight1}.

\begin{theorem}
  \label{thm:main}Let the initial values $(u_0^{\varepsilon})_{\varepsilon \in
  (0, 1)}$ and $(v_0^{\varepsilon})_{\varepsilon \in (0, 1)}$ be given
  {so that both $(u_0^{\varepsilon})_{\varepsilon \in (0, 1)}$
  and $(\varepsilon^{1 / 2} v^{\varepsilon}_0)_{\varepsilon \in (0, 1)}$ are}
  bounded in $H$ uniformly in $\varepsilon$. There exists a family
  $((\Omega^{\varepsilon}, \mathcal{F}^{\varepsilon},
  (\mathcal{F}^{\varepsilon}_t)_{t \geqslant 0}, \tmmathbf{P}^{\varepsilon}),
  u^{\varepsilon}, r^{\varepsilon}, W^{\varepsilon})$, $\varepsilon \in (0,
  1),$ of probabilistically weak rough path solutions to \eqref{eq:u},
  \eqref{eq:w}, \eqref{eq:r}, such that
  \begin{equation}
    \| u^{\varepsilon} \|_{L^{\infty}_{T} H} + \| u^{\varepsilon} \|_{L^2_{T} H^1}
    \lesssim 1 \quad \tmmathbf{P}^{\varepsilon} \text{-a.s.\quad and} \qquad
    \mathbb{E}^{\varepsilon} [\| r^{\varepsilon} \|_{L^2_{T} H^{\gamma}}^2]
    \lesssim 1 \label{eq:bd}
  \end{equation}
  with implicit constants independent of $\varepsilon$.
  
  For every such family of probabilistically weak rough path solutions there
  exists a subsequence, still denoted by $\varepsilon \rightarrow 0$, such
  that $(u^{\varepsilon}, (- C)^{- 1} Q^{1 / 2} W^{\varepsilon})$ converges in
  law to a probabilistically weak rough path solution to \eqref{eq:ulim}.
\end{theorem}

In view of the above discussion, the existence of a probabilistically weak
rough path solution to \eqref{eq:u}, \eqref{eq:w}, \eqref{eq:r} for every
$\varepsilon \in (0, 1)$ can be proved by classical arguments via Galerkin
approximation together with the stochastic compactness method. We omit the
details of the construction but we derive the uniform estimate \eqref{eq:bd}
in Section~\ref{s:bdu} and Section~\ref{s:bdr}.

\begin{remark}
  \label{r:pw}Furthermore, we note that by Skorokhod representation theorem,
  for every family $$((\Omega^{\varepsilon}, \mathcal{F}^{\varepsilon},
  (\mathcal{F}^{\varepsilon}_t)_{t \geqslant 0}, \tmmathbf{P}^{\varepsilon}),
  u^{\varepsilon}, r^{\varepsilon}, W^{\varepsilon}),\qquad \varepsilon \in (0,
  1),$$  of probabilistically weak rough path solutions to \eqref{eq:u},
  \eqref{eq:w}, \eqref{eq:r}, there exists a sequence $\varepsilon \rightarrow
  0$ and a family $((\bar{\Omega}, \bar{\mathcal{F}},
  (\bar{\mathcal{F}}^{\varepsilon}_t)_{t \geqslant 0},
  \overline{\tmmathbf{P}}), \bar{u}^{\varepsilon}, \bar{r}^{\varepsilon},
  \bar{W}^{\varepsilon})$ of probabilistically weak rough path solutions such
  that the laws of $(u^{\varepsilon}, r^{\varepsilon}, W^{\varepsilon})$ and
  $(\bar{u}^{\varepsilon}, \bar{r}^{\varepsilon}, \bar{W}^{\varepsilon})$
  coincide. Thus, without loss of generality, we assume from now on that the
  approximate solutions $(u^{\varepsilon}, r^{\varepsilon}, W^{\varepsilon})$
  in Theorem~\ref{thm:main} are defined on a common probability space
  $(\Omega, \mathcal{F}, \tmmathbf{P})$.
\end{remark}

\section{Ergodicity of the Ornstein--Uhlenbeck process}\label{s:erg}

In Section~\ref{s:2} as well as in Section~\ref{s:av}, we will need a suitable
ergodic theorem in order to deduce the convergence of the second component of
the rough path $Y^{\varepsilon, 2}$ and the convergence towards the
It{\^o}--Stokes velocity $\bar{r}$ given by \eqref{eq:itostokes}. More
precisely, letting $\tilde{w}_t = w^{\varepsilon}_{\varepsilon t}$ both
problems reduce to the convergence of the ergodic average
\[ \lim_{t \rightarrow \infty} \frac{1}{t} \int_0^t F (\tilde{w}_s) d s = \int
   F (w) d \mu (w), \]
where $\mu$ is the (unique) invariant measure for $\tilde{w}$ and $F$ is a
certain quadratic function. We observe that $\tilde{w}$ solves
\begin{equation}
  d \tilde{w} = C \tilde{w} d t + Q^{1 / 2} d \tilde{W} \label{eq:tildew}
\end{equation}
with some cylindrical Wiener process $\tilde{W}$. As a matter of fact, since
in view of Remark~\ref{r:pw}, $w^{\varepsilon}$ actually satisfies
\eqref{eq:w} with a Wiener process $W^{\varepsilon}$, the rescaling
$\tilde{W}$ also depends on $\varepsilon$. However, this dependence is
irrelevant for the sequel and we drop it for notational simplicity, as we only
use the law of $\tilde{W}$ in our arguments.

We denote by $\tilde{w} (w)$ the unique solution to \eqref{eq:tildew} with the
initial condition $w \in H$. It is Gaussian and if $w \in H^{\sigma}$ then in
view of \eqref{eq:6} and \eqref{eq:tr} it satisfies the global-in-time
estimate
\begin{equation}
  \sup_{t \in [0, \infty)} \mathbb{E} [\| \tilde{w}_t (w) \|^2_{H^{\sigma}}]
  \lesssim \| w \|^2_{H^{\sigma}} + \int_0^{\infty} \| (- \Delta)^{\sigma / 2}
  e^{C t} Q^{1 / 2} \|_{L_2 (H)}^2 d t < \infty . \label{est:tilew2}
\end{equation}
This can be obtained directly from the mild formulation. Moreover,
\eqref{eq:tildew} generates a Markov process with the unique invariant measure
$\mu =\mathcal{N} (0, Q_{\infty})$ with $Q_{\infty}$ defined in \eqref{eq:mu},
see e.g. Theorem~11.17, Theorem 11.20 {\cite{DPZa14}}.

The desired ergodic theorem needed for Section~\ref{s:2} reads as follows.

\begin{proposition}
  \label{p:erg1}Let $F : H^{\sigma} \rightarrow H^{\sigma} \otimes H^{\sigma}$
  be quadratic in the sense that $F (w) = a_1 (w) \otimes a_2 (w)$ for some
  bounded linear operators $a_1, a_2 : H^{\sigma} \rightarrow H^{\sigma}
  \otimes H^{\sigma}$. Let $\tilde{w}$ be a solution to \eqref{eq:tildew} with
  an initial condition $\tilde{w}_0 \in H^{\sigma}$. Then for all $t \in [0,
  \infty)$
  \begin{equation}
    \left\| \mathbb{E} [F (\tilde{w}_t)] - \int F (w) d \mu (w)
    \right\|_{H^{\sigma} \otimes H^{\sigma}} \lesssim e^{- \iota t} (1 + \|
    \tilde{w}_0 \|_{H^{\sigma}}^2), \qquad t \geqslant 0, \label{eq:expmix}
  \end{equation}
\rmb{where $\iota$ is the parameter from assumption \eqref{eq:6},} and
  \begin{equation}
    \mathbb{E} \left[ \left\| \frac{1}{t} \int_0^t F (\tilde{w}_s) d s - \int
    F (w) d \mu (w) \right\|^2_{H^{\sigma} \otimes H^{\sigma}} \right]
    \lesssim \frac{1}{t} (1 + \| \tilde{w}_0 \|_{H^{\sigma}}^4) .
    \label{eq:erg}
  \end{equation}
\end{proposition}

\begin{proof}
  {Let $(f_k)_{k \in \mathbb{N}}$ be an orthonormal basis of
  $H^{\sigma}$.} Recall that $\tilde{w}$ is Markov and denote its Markov
  semigroup by $(P_t)_{t \geqslant 0}$. It acts generally on a function
  \[ H^{\sigma} \rightarrow H^{\sigma} \otimes H^{\sigma}, \qquad w \mapsto
     \varphi (w) = \sum_{k, \ell \in \mathbb{N}} \langle \varphi (w), f_k
     \otimes f_{\ell} \rangle_{H^{\sigma} \otimes H^{\sigma}} f_k \otimes
     f_{\ell} \]
  component-wise as
  \[ (P_t \varphi) (w) \assign \mathbb{E} [\varphi (\tilde{w}_t (w))] =
     \sum_{k, \ell \in \mathbb{N}} \mathbb{E} [\langle \varphi (\tilde{w}_t
     (w)), f_k \otimes f_{\ell} \rangle_{H^{\sigma} \otimes H^{\sigma}}] f_k
     \otimes f_{\ell}, \]
  where $\tilde{w}_t (w)$ denotes the solution to \eqref{eq:tildew} with the
  initial condition $w \in H^{\sigma}$. In order to prove \eqref{eq:expmix} we
  now write
  \[ \mathbb{E} [F (\tilde{w}_t)] - \int F (w) d \mu (w) =\mathbb{E} [F
     (\tilde{w}_t)] - \int (P_t F) (w) d \mu (w) = \int \mathbb{E} [F
     (\tilde{w}_t) - F (\tilde{w}_t (w))] d \mu (w) \]
  \[ = \int \mathbb{E} [a_1 (\tilde{w}_t) \otimes a_2 (\tilde{w}_t -
     \tilde{w}_t (w)) + a_1 (\tilde{w}_t - \tilde{w}_t (w)) \otimes a_2
     (\tilde{w}_t (w))] d \mu (w) \]
  and using the fact that $\tilde{w}_t - \tilde{w}_t (w) = e^{C t}
  (\tilde{w}_0 - w)$ we further obtain
  \[ = e^{- \iota t} \int \mathbb{E} [a_1 (\tilde{w}_t) \otimes a_2 (e^{\iota
     t} e^{C t} (\tilde{w}_0 - w)) + a_1 (e^{\iota t} e^{C t} (\tilde{w}_0 -
     w)) \otimes a_2 (\tilde{w}_t (w))] d \mu (w) . \]
  Since the above integral is controlled by the assumption on $a_1, a_2$,
  \eqref{eq:6} and the integrability of the invariant measure $\mu$ as
  \[ \left\| \int \mathbb{E} \left[ a_1 (\tilde{w}_t) \otimes a_2 (e^{\iota t}
     e^{C t} (\tilde{w}_0 - w)) + a_1 (e^{\iota t} e^{C t} (\tilde{w}_0 - w))
     \otimes a_2 (\tilde{w}_t (w)) \right] d \mu (w) \right\|_{H^{\sigma}
     \otimes H^{\sigma}} \]
  \[ \lesssim \int \mathbb{E} [\| \tilde{w}_t \|_{H^{\sigma}} \| \tilde{w}_0 -
     w \|_{H^{\sigma}} + \| \tilde{w}_0 - w \|_{H^{\sigma}} \| \tilde{w}_t (w)
     \|_{H^{\sigma}}] d \mu (w) \lesssim (1 + \| \tilde{w}_0
     \|_{H^{\sigma}}^2), \]
  \eqref{eq:expmix} follows.
  
  For \eqref{eq:erg}, we first define $\bar{F} \assign \int F (w) d \mu (w)$
  and write
  \[ \mathbb{E} \left[ \left\| \frac{1}{t} \int_0^t F (\tilde{w}_s) d s -
     \bar{F} \right\|_{H^{\sigma} \otimes H^{\sigma}}^2 \right] =
     \frac{2}{t^2} \int_0^t \int_r^t \mathbb{E} [\langle F (\tilde{w}_s) -
     \bar{F}, F (\tilde{w}_r) - \bar{F} \rangle_{H^{\sigma} \otimes
     H^{\sigma}}] d s d r. \]
  Since $r \leqslant s$, it holds by the Markov property
  \[ \mathbb{E} [\langle F (\tilde{w}_s) - \bar{F}, F (\tilde{w}_r) - \bar{F}
     \rangle_{H^{\sigma} \otimes H^{\sigma}}] =\mathbb{E} [\langle F
     (\tilde{w}_r) - \bar{F}, \mathbb{E} [F (\tilde{w}_s) - \bar{F} |
     \mathcal{F}_r \nobracket] \rangle_{H^{\sigma} \otimes H^{\sigma}}] \]
  \[ =\mathbb{E} [\langle F (\tilde{w}_r) - \bar{F}, (P_{s - r} (F (\cdummy) -
     \bar{F})) (\tilde{w}_r) \rangle_{H^{\sigma} \otimes H^{\sigma}}] \]
  \[ \leqslant (\mathbb{E} [\| F (\tilde{w}_r) - \bar{F} \|_{H^{\sigma}
     \otimes H^{\sigma}}^2])^{1 / 2} (\mathbb{E} [\| (P_{s - r} (F (\cdummy) -
     \bar{F})) (\tilde{w}_r) \|_{H^{\sigma} \otimes H^{\sigma}}^2])^{1 / 2} .
  \]
  By assumption on $F$ and \eqref{est:tilew2} we have
  \[ \sup_{r \in [0, \infty)} \mathbb{E} [\| F (\tilde{w}_r) - \bar{F}
     \|_{H^{\sigma} \otimes H^{\sigma}}^2] \lesssim 1 +\mathbb{E} [\|
     \tilde{w}_r \|_{H^{\sigma}}^4] \lesssim 1 + \| \tilde{w}_0
     \|_{H^{\sigma}}^4, \]
  whereas by \eqref{eq:expmix} and \eqref{est:tilew2} we obtain
  \[ \mathbb{E} [\| (P_{s - r} (F (\cdummy) - \bar{F})) (\tilde{w}_r)
     \|_{H^{\sigma} \otimes H^{\sigma}}^2] \lesssim e^{- 2 \iota (s - r)} (1
     +\mathbb{E} [\| \tilde{w}_r \|_{H^m}^2])^2 \lesssim e^{- 2 \iota (s - r)}
     (1 + \| \tilde{w}_0 \|_{H^{\sigma}}^4) . \]
  Therefore,
  \[ \mathbb{E} \left[ \left\| \frac{1}{t} \int_0^t F (\tilde{w}_s) d s -
     \bar{F} \right\|_{H^{\sigma} \otimes H^{\sigma}}^2 \right] \lesssim
     \frac{1}{t^2} \int_0^t \int_r^t e^{- \iota (s - r)} d s d r (1 + \|
     \tilde{w}_0 \|_{H^{\sigma}}^4) \lesssim \frac{1}{t} (1 + \| \tilde{w}_0
     \|_{H^{\sigma}}^4) \]
  which completes the proof.
\end{proof}

The same proof applies to other quadratic functions $F$. For instance, in
Section~\ref{s:av} we make use the following variant.

\begin{corollary}
  \label{p:erg}Let $F : H^{\sigma} \rightarrow H$ be quadratic in the sense
  that there exists a bounded bilinear operators $a : H^{\sigma} \times
  H^{\sigma} \rightarrow H$ such that $F (w) = a (w, w)$. Let $\tilde{w}$ be a
  solution to \eqref{eq:tildew} with an initial condition $\tilde{w}_0 \in
  H^{\sigma}$. Then for all $t \in [0, \infty)$
  \[ \mathbb{E} \left[ \left\| \frac{1}{t} \int_0^t F (\tilde{w}_s) d s - \int
     F (w) d \mu (w) \right\|^2_H \right] \lesssim \frac{1}{t} (1 + \|
     \tilde{w}_0 \|_{H^{\sigma}}^4) . \]
\end{corollary}

\section{Convergence of the rough driver}\label{s:2}

The goal of this section is to prove Theorem~\ref{l:11}. 

\begin{proof}
  We intend to apply the Kolmogorov criterion for rough path distance,
  Theorem~3.3 in {\cite{FH14}}, to deduce that for all $q \in [2, \infty)$,
  $\theta \in (2 / 3, 1]$ and $\beta \in (1 / 3, \theta / 2)$
  \begin{equation}
    (\mathbb{E} [\| Y^{\varepsilon, 1} - B^1 \|_{C^{\beta}_2 ([0, T] ;
    H^{\sigma})}^q])^{1 / q} + (\mathbb{E} [\| Y^{\varepsilon, 2} - B^2
    \|_{C^{2 \beta}_2 ([0, T] ; H^{\sigma} \otimes H^{\sigma})}^q])^{1 / q}
    \lesssim \varepsilon^{(1 - \theta) / 2} . \label{eq:bound}
  \end{equation}
  This then readily implies the first statement of the theorem. First, we
  observe that the processes $w^{\varepsilon}, y^{\varepsilon}, B$ are jointly
  Gaussian as being linear combinations of $W$. Accordingly, $Y^{\varepsilon,
  1} - B^1$ and $Y^{\varepsilon, 2} - B^2$, respectively, live in the Wiener
  chaos of order 1 and 2, respectively. Hence all the $L^q (\Omega)$-norms are
  equivalent and by Theorem~3.3 in {\cite{FH14}} it is enough to prove for \rmb{all}
  $\theta \in [0, 1]$ 
  \begin{equation}
    \mathbb{E} [\| Y_{s t}^{\varepsilon, 1} - B_{s t}^1 \|^2_{H^{\sigma}}]
    \lesssim \varepsilon^{1 - \theta} | t - s |^{\theta}, \qquad \mathbb{E}
    [\| Y^{\varepsilon, 2}_{s t} - B^2_{s t} \|^2_{H^{\sigma} \otimes
    H^{\sigma}}] \lesssim \varepsilon^{1 - \theta} | t - s |^{2 \theta} .
    \label{eq:yy7}
  \end{equation}

  Step 1: Regarding the first bound in \eqref{eq:yy7}, we note that according
  to \eqref{eq:89} and \eqref{eq:tr} $B$ is an $H^{\sigma}$-valued Wiener
  process and in view of \eqref{eq:w} and \eqref{eq:y} it holds for $t \in [0,
  T]$
  \begin{equation}
    y^{\varepsilon}_t = B_t - \varepsilon^{1 / 2} (- C)^{- 1}
    w^{\varepsilon}_t . \label{eq:oo}
  \end{equation}
  Consequently, by \eqref{eq:89} and \eqref{est:tilew2} it holds for all $t
  \in [0, T]$ (using again the notation $\tilde{w}_t = w_{\varepsilon
  t}^{\varepsilon}$)
  \begin{equation}
    \mathbb{E} [\| y^{\varepsilon}_t - B_t \|_{H^{\sigma}}^2] = \varepsilon
    \mathbb{E} [\| (- C)^{- 1} w^{\varepsilon}_t \|_{H^{\sigma}}^2] =
    \varepsilon \mathbb{E} [\| (- C)^{- 1} \tilde{w}_{\varepsilon^{- 1} t}
    \|_{H^{\sigma}}^2] \lesssim \varepsilon . \label{eq:888}
  \end{equation}
  \rmb{This proves the first bound in \eqref{eq:yy7} for $\theta=0$.}
  On the other hand, \rmb{aiming for the first bound in \eqref{eq:yy7} with $\theta=1$,} we write
  \begin{equation}
    \mathbb{E} [\| \delta y_{s t}^{\varepsilon} - \delta B_{s t}
    \|_{H^{\sigma}}^2] = \varepsilon \mathbb{E} [\| (- C)^{- 1} \delta w_{s
    t}^{\varepsilon} \|_{H^{\sigma}}^2] = \varepsilon \mathbb{E} [\| (-
    C)^{- 1} \delta \tilde{w}_{\varepsilon^{- 1} s, \varepsilon^{- 1} t}
    \|^2_{H^{\sigma}}]  \label{eq:dd}
  \end{equation}
  and observe that
  \[ \delta \tilde{w}_{s t} = (e^{C t} - e^{C s}) \tilde{w}_0 + (e^{C (t - s)}
     - \tmop{Id}) \int_0^s e^{C (s - r)} Q^{1 / 2} d \tilde{W}_r + \int_s^t
     e^{C (t - r)} Q^{1 / 2} d \tilde{W}_r \]
  \[ = (e^{C (t - s)} - \tmop{Id}) \tilde{w}_s + \int_s^t e^{C (t - r)} Q^{1 /
     2} d \tilde{W}_r, \]
  which implies by \eqref{eq:89}
  \[ \mathbb{E} [\| (- C)^{- 1} \delta \tilde{w}_{s t} \|_{H^{\sigma}}^2]
     \lesssim \mathbb{E} [\| (- C)^{- 1} (e^{C (t - s)} - \tmop{Id})
     \tilde{w}_s \|_{H^{\sigma}}^2] +\mathbb{E} \left[ \left\| \int_s^t e^{C
     (t - r)} Q^{1 / 2} d \tilde{W}_r \right\|_{H^{\sigma}}^2 \right] . \]
  Here the second term is bounded using \eqref{eq:tr} as
  \[ \mathbb{E} \left[ \left\| \int_s^t e^{C (t - r)} Q^{1 / 2} d \tilde{W}_r
     \right\|_{H^{\sigma}}^2 \right] =\mathbb{E} \int_s^t \| (-
     \Delta)^{\sigma / 2} e^{C (t - r)} Q^{1 / 2} \|_{L_2 (H)}^2 d r \lesssim
     | t - s | . \]
  For the first one, we use \eqref{eq:77} and \eqref{est:tilew2} to get
  \[ \mathbb{E} [\| (- C)^{- 1} (e^{C (t - s)} - \tmop{Id}) \tilde{w}_s \|_{H^{\sigma}}^2] \lesssim | t - s | \mathbb{E} \left[ \left\|  \tilde{w}_s \right\|^2_{H^{\sigma}}  \right] \lesssim | t - s | . \]
  Plugging this into \eqref{eq:dd} yields
  \begin{equation}
    \mathbb{E} [\| \delta y_{s t}^{\varepsilon} - \delta B_{s t}
    \|_{H^{\sigma}}^2] \lesssim | t - s | . \label{eq:ddd}
  \end{equation}
  \rmb{In other words, the first bound in \eqref{eq:yy7} holds with $\theta=1$.
Finally, let $\theta\in[0,1]$ be arbitrary. Then by interpolation,} \eqref{eq:888} and \eqref{eq:ddd} imply
  \[ \mathbb{E} [\| Y_{s t}^{\varepsilon, 1} - B_{s t}^1 \|^2_{H^{\sigma}}]
     \leqslant (\mathbb{E} [\| \delta y_{s t}^{\varepsilon} - \delta B_{s t}
     \|_{H^{\sigma}}^2])^{1 - \theta} {{(\mathbb{E} [\| \delta y_{s
     t}^{\varepsilon} - \delta B_{s t} \|_{H^{\sigma}}^2])^{\theta}} } 
     \lesssim \varepsilon^{1 - \theta} | t - s |^{\theta} \]
  and the first bound in \eqref{eq:yy7} is verified.
  
  As in Theorem~3.3 in {\cite{FH14}}, this readily implies
  \[ (\mathbb{E} [\| Y^{\varepsilon, 1} - B^1 \|_{C^{\beta}_2 ([0, T] ;
     H^{\sigma})}^q])^{1 / q} \lesssim \varepsilon^{(1 - \theta) / 2} \]
  and in particular also proves the convergence of $Y^{\varepsilon, 1}
  \rightarrow B^1$ in $L^q (\Omega ; C^{\beta}_2 ([0, T] ; H^{\sigma}))$.
  
  Step 2: To establish the second bound in \eqref{eq:yy7} we proceed similarly
  as in Step 1. Towards the interpolation argument, we first show the bound
  for $\theta = 1$, afterwards for $\theta = 0$.
  
  Step 2.1: For $\theta = 1$, we recall that $B^2$ is a rough path lift of $B$
  defined in \eqref{eq:B2ito}. Here the martingale part is given by It{\^o}'s
  iterated integral, which is well-defined in $H^{\sigma} \otimes H^{\sigma}$
  due to \eqref{eq:89} and \eqref{eq:tr}, the finite variation part is
  controlled due to the properties of the invariant measure $\mu$. Therefore
  \[ \mathbb{E} [\| B^2_{s t} \|^2_{H^{\sigma} \otimes H^{\sigma}}] \lesssim |
     t - s |^2 . \]
  Hence, we shall prove
  \begin{equation}
    \mathbb{E} [\| Y^{\varepsilon, 2}_{s t} \|^2_{H^{\sigma} \otimes
    H^{\sigma}}] \lesssim | t - s |^2 . \label{eq:mmm23}
  \end{equation}
  To this end, we apply \eqref{eq:oo} to get
  \[ Y^{\varepsilon, 2}_{s t} = \int_s^t \delta y^{\varepsilon}_{s r} \otimes
     d y^{\varepsilon}_r = \int_s^t \delta y^{\varepsilon}_{s r} \otimes d B_r
     - \varepsilon^{1 / 2} \int_s^t \delta y^{\varepsilon}_{s r} \otimes (-
     C)^{- 1} d w_r^{\varepsilon} . \]
  Integrating by parts in the second term yields
  \[ \varepsilon^{1 / 2} \int_s^t \delta y^{\varepsilon}_{s r} \otimes (-
     C)^{- 1} d w_r^{\varepsilon} = \varepsilon^{1 / 2} \int_s^t
     y^{\varepsilon}_r \otimes (- C)^{- 1} d w_r^{\varepsilon} -
     \varepsilon^{1 / 2} y^{\varepsilon}_s \otimes (- C)^{- 1} \delta
     w^{\varepsilon}_{s t} \]
  \[ = - \varepsilon^{1 / 2} \int_s^t d y^{\varepsilon}_r \otimes (- C)^{- 1}
     w^{\varepsilon}_r + \varepsilon^{1 / 2} [y^{\varepsilon}_r \otimes (-
     C)^{- 1} w^{\varepsilon}_r]_{r = s}^{r = t} - \varepsilon^{1 / 2}
     y^{\varepsilon}_s \otimes (- C)^{- 1} \delta w^{\varepsilon}_{s t} \]
  \[ = - \varepsilon^{1 / 2} \int_s^t d y^{\varepsilon}_r \otimes (- C)^{- 1}
     w^{\varepsilon}_r + \varepsilon^{1 / 2} y^{\varepsilon}_t \otimes (-
     C)^{- 1} w^{\varepsilon}_t - \varepsilon^{1 / 2} y^{\varepsilon}_s
     \otimes (- C)^{- 1} w^{\varepsilon}_t \]
  \[ = - \varepsilon^{1 / 2} \int_s^t d y^{\varepsilon}_r \otimes (- C)^{- 1}
     w^{\varepsilon}_r + \varepsilon^{1 / 2} \delta y^{\varepsilon}_{s t}
     \otimes (- C)^{- 1} w^{\varepsilon}_t \]
  \[ = - \int_s^t w^{\varepsilon}_r \otimes (- C)^{- 1} w^{\varepsilon}_r d r
     + \int_s^t w^{\varepsilon}_r d r \otimes (- C)^{- 1} w^{\varepsilon}_t .
  \]
  We estimate each term separately. The first term is controlled by It{\^o}'s
  isometry as
  \[ \mathbb{E} \left[ \left\| \int_s^t \delta y^{\varepsilon}_{s r} \otimes d
     B_r \right\|^2_{H^{\sigma} \otimes H^{\sigma}} \right] =\mathbb{E} \left[
     \left\| \int_s^t \delta y^{\varepsilon}_{s r} \otimes (- C)^{- 1} Q^{1 /
     2} d W_r \right\|^2_{H^{\sigma} \otimes H^{\sigma}} \right] \]
  \[ =\mathbb{E} \left[ \int_s^t \| \delta y^{\varepsilon}_{s r} \otimes (-
     C)^{- 1} Q^{1 / 2} \cdummy \|_{L_2 (H, H^{\sigma} \otimes H^{\sigma})}^2
     d r \right], \]
  where
  \[ \| \delta y^{\varepsilon}_{s r} \otimes (- C)^{- 1} Q^{1 / 2} \cdummy
     \|_{L_2 (H, H^{\sigma} \otimes H^{\sigma})}^2 = \sum_{k \in \mathbb{N}}
     \| \delta y^{\varepsilon}_{s r} \otimes (- C)^{- 1} Q^{1 / 2} g_k
     \|_{H^{\sigma} \otimes H^{\sigma}}^2 \]
  \[ = \| \delta y^{\varepsilon}_{s r} \|_{H^{\sigma}}^2 \sum_{k \in
     \mathbb{N}} \| (- C)^{- 1} Q^{1 / 2} e_k \|_{H^{\sigma}}^2 = \| \delta
     y^{\varepsilon}_{s r} \|_{H^{\sigma}}^2 \| (- C)^{- 1} Q^{1 / 2}
     \|^2_{L_2 (H ; H^{\sigma})} . \]
  Hence by \eqref{eq:89} and \eqref{eq:tr} we deduce
  \[ \mathbb{E} \left[ \left\| \int_s^t \delta y^{\varepsilon}_{s r} \otimes d
     B_r \right\|^2_{H^{\sigma} \otimes H^{\sigma}} \right] \lesssim
     \mathbb{E} \left[ \int_s^t \| \delta y^{\varepsilon}_{s r}
     \|_{H^{\sigma}}^2 d r \right] \lesssim \int_s^t (r - s) d r \lesssim | t
     - s |^2, \]
  where we used the estimate for the first component proved in Step~1 in the
  second to last inequality.
  
  For the second term, we use \eqref{eq:89} together with \eqref{est:tilew2}
  to obtain
  \[ \mathbb{E} \left[ \left\| \int_s^t w^{\varepsilon}_r \otimes (- C)^{- 1}
     w^{\varepsilon}_r d r \right\|^2_{H^{\sigma} \otimes H^{\sigma}} \right]
     \leqslant | t - s | \int_s^t \mathbb{E} [\| w^{\varepsilon}_r
     \|^4_{H^{\sigma}}] d r \lesssim | t - s |^2 . \]
  Finally, by the same assumptions we bound the last term as
  \[ \mathbb{E} \left[ \left\| \int_s^t w^{\varepsilon}_r d r \otimes (- C)^{-
     1} w^{\varepsilon}_t \right\|^2_{H^{\sigma} \otimes H^{\sigma}} \right]
     \lesssim | t - s | \int_s^t \mathbb{E} [\| w^{\varepsilon}_r
     \|^2_{H^{\sigma}} \| w^{\varepsilon}_t \|^2_{H^{\sigma}}] d r \lesssim |
     t - s |^2 . \]
  This completes the proof of \eqref{eq:mmm23} and consequently the second
  bound \eqref{eq:yy7} with $\theta = 1$ follows.
  
  Step 2.2: We aim at proving the second bound in \eqref{eq:yy7} with $\theta
  = 0$. Proceeding as in the analogous part of Step 1, we want a rate of
  convergence of the form
  \begin{equation}
    \sup_{s, t \in [0, T]} \mathbb{E} [\| Y_{s t}^{\varepsilon, 2} - B_{s t}^2
    \|^2_{H^{\sigma} \otimes H^{\sigma}}]  \lesssim \varepsilon .
    \label{eq:mmm2}
  \end{equation}
  To this end, it is enough to prove for all $t \in [0, T]$
  \begin{equation}
    \mathbb{E} [\| Y_{0 t}^{\varepsilon, 2} - B_{0 t}^2 \|^2_{H^{\sigma}
    \otimes H^{\sigma}}]  \lesssim \varepsilon . \label{eq:0t}
  \end{equation}
  Indeed, by Chen's relation for $0 \leqslant r \leqslant t$
  \[ \delta (Y^{\varepsilon, 2} - B^2)_{0 r t} = (Y_{0 t}^{\varepsilon, 2} -
     B_{0 t}^2) - (Y_{0 r}^{\varepsilon, 2} - B_{0 r}^2) - (Y_{r
     t}^{\varepsilon, 2} - B_{r t}^2) = Y^{\varepsilon, 1}_{0 r} \otimes
     Y^{\varepsilon, 1}_{r t} - B^1_{0 r} \otimes B^1_{r t} \]
  \[ = (Y^{\varepsilon, 1}_{0 r} - B^1_{0 r}) \otimes Y^{\varepsilon, 1}_{r t}
     + B^1_{0 r} \otimes (Y^{\varepsilon, 1}_{r t} - B^1_{r t}), \]
  meaning, if we already know the bound \eqref{eq:0t} for all $t$ we may
  combine it with \eqref{eq:888} to get \eqref{eq:mmm2}.
  
  Let us fix $t \in [0, T]$ and prove \eqref{eq:0t}. We recall that by
  \eqref{eq:oo} it holds
  \[ Y_{0 t}^{\varepsilon, 2} = \int_0^t y^{\varepsilon}_s \otimes d
     y_s^{\varepsilon} = \int_0^t y^{\varepsilon}_s \otimes d B_s -
     \varepsilon^{1 / 2} \int_0^t y^{\varepsilon}_s \otimes (- C)^{- 1} d
     w_s^{\varepsilon} . \]
  Integrating by parts yields
  \[ Y_{0 t}^{\varepsilon, 2} = \int_0^t y^{\varepsilon}_s \otimes d B_s -
     \varepsilon^{1 / 2} y^{\varepsilon}_t \otimes (- C)^{- 1}
     w^{\varepsilon}_t + \varepsilon^{1 / 2} \int_0^t d y^{\varepsilon}_s
     \otimes (- C)^{- 1} w_s^{\varepsilon} \]
  \begin{equation}
    = \int_0^t y^{\varepsilon}_s \otimes d B_s - \varepsilon^{1 / 2}
    y^{\varepsilon}_t \otimes (- C)^{- 1} w^{\varepsilon}_t + \int_0^t
    w^{\varepsilon}_s \otimes (- C)^{- 1} w_s^{\varepsilon} d s. \label{eq:15}
  \end{equation}
  In order to treat the last term, we apply the ergodic theorem from
  Proposition~\ref{p:erg1}. It shall converge to the finite variation term in
  the It{\^o} expression of $B^2$, i.e. \eqref{eq:B2ito}. Namely, we need to
  bound
  \[ \mathbb{E} \left[ \left\| \int_0^t w^{\varepsilon}_s \otimes (- C)^{- 1}
     w_s^{\varepsilon} d s - t \int w \otimes (- C)^{- 1} w d \mu (w)
     \right\|^2_{H^{\sigma} \otimes H^{\sigma}} \right], \]
  which by the change of variables $r = \varepsilon^{- 1} s$, $d r =
  \varepsilon^{- 1} d s$, $w^{\varepsilon}_s = \tilde{w}_{\varepsilon^{- 1}
  s}$ and Proposition~\ref{p:erg1} rewrites as
  \[ =\mathbb{E} \left[ \left\| \varepsilon \int_0^{\varepsilon^{- 1} t}
     \tilde{w}_r \otimes (- C)^{- 1} \tilde{w}_r d r - t \int w \otimes (-
     C)^{- 1} w d \mu (w) \right\|^2_{H^{\sigma} \otimes H^{\sigma}} \right]
  \]
  \[ = t\mathbb{E} \left[ \left\| \frac{1}{\varepsilon^{- 1} t}
     \int_0^{\varepsilon^{- 1} t} \tilde{w}_r \otimes (- C)^{- 1} \tilde{w}_r
     d r - \int w \otimes (- C)^{- 1} w d \mu (w) \right\|^2_{H^{\sigma}
     \otimes H^{\sigma}} \right] \lesssim \varepsilon . \]

  For the convergence of the stochastic integral in \eqref{eq:15}, we have by
  It{\^o}'s isometry, \eqref{eq:tr} and \eqref{eq:888}
  \[ \mathbb{E} \left[ \left\| \int_0^t y^{\varepsilon}_s \otimes d B_s -
     \int_0^t B_s \otimes d B_s \right\|^2_{H^{\sigma} \otimes H^{\sigma}}
     \right] \lesssim \mathbb{E} \left[ \int_0^t \| y^{\varepsilon}_s - B_s
     \|_{H^{\sigma}}^2 d s \right] \lesssim \varepsilon . \]
  The remaining term in \eqref{eq:15} is controlled as follows
  \[ \mathbb{E} [\| \varepsilon^{1 / 2} y^{\varepsilon}_t \otimes (- C)^{- 1}
     w^{\varepsilon}_t \|^2_{H^{\sigma} \otimes H^{\sigma}}] = \varepsilon 
     \mathbb{E} [\| y^{\varepsilon}_t \|^2_{H^{\sigma}} \| (- C)^{- 1}
     w^{\varepsilon}_t \|^2_{H^{\sigma}}] \]
  \[ \leqslant \varepsilon (\mathbb{E} [\| y^{\varepsilon}_t
     \|_{H^{\sigma}}^4])^{1 / 2} (\mathbb{E} [\| w^{\varepsilon}_t
     \|^4_{H^{\sigma}}])^{1 / 2} \lesssim \varepsilon, \]
  where we used Gaussianity and \eqref{eq:888} to get
  \[ \mathbb{E} [\| y^{\varepsilon}_t \|_{H^{\sigma}}^4] \lesssim \mathbb{E}
     [\| y^{\varepsilon}_t - B_t \|_{H^{\sigma}}^4] +\mathbb{E} [\| B_t
     \|_{H^{\sigma}}^4] \lesssim 1. \]

  This completes the proof of \eqref{eq:0t} where $B^2$ is given by the
  It{\^o} formulation \eqref{eq:B2ito}. It also completes \ the proof of the
  second bound in \eqref{eq:yy7} for $\theta = 0$. Interpolating with the
  bound for $\theta = 1$, the general version of \eqref{eq:yy7} follows.
  
  Step 2.3: Now, it remains to verify the Stratonovich formulation
  \eqref{eq:B2str}. Rewriting It{\^o}'s stochastic integral into
  Stratonovich's form we obtain
  \begin{equation}
    B^2_{0 t} = \int_0^t B_s \otimes \circ d B_s - \frac{t}{2} \sum_{m \in
    \mathbb{N}} (- C)^{- 1} Q^{1 / 2} e_m \otimes (- C)^{- 1} Q^{1 / 2} e_m +
    t \int w \otimes (- C)^{- 1} w d \mu (w) . \label{eq:11}
  \end{equation}
  We intend to prove that the symmetric part of the very last term precisely
  cancels with the It{\^o}--Stratonovich corrector, i.e. the second term in
  \eqref{eq:11}. To this end, we first observe that the It{\^o}--Stratonovich
  corrector projected at $e_k \otimes e_{\ell}$ reads as
  \[ - \frac{t}{2} \left\langle \sum_{m \in \mathbb{N}} (- C)^{- 1} Q^{1 / 2}
     e_m \otimes (- C)^{- 1} Q^{1 / 2} e_m, e_k \otimes e_{\ell} \right\rangle
  \]
  \[ = - \frac{t}{2} \sum_{m \in \mathbb{N}} \langle (- C)^{- 1} Q^{1 / 2}
     e_m, e_k \rangle \langle (- C)^{- 1} Q^{1 / 2} e_m, e_{\ell} \rangle \]
  \[ = - \frac{t}{2} \langle Q^{1 / 2} ((- C)^{- 1})^{\ast} e_k, Q^{1 / 2} ((-
     C)^{- 1})^{\ast} e_{\ell} \rangle = - \frac{t}{2} \langle (- C)^{- 1} Q
     ((- C)^{- 1})^{\ast} e_k, e_{\ell} \rangle . \]
  We perform the same projection to the last term in \eqref{eq:11}. Since $\mu
  =\mathcal{N} (0, Q_{\infty})$ with $Q_{\infty}$ given in \eqref{eq:mu} and
  since it holds for any $g, h \in H$
  \[ \langle Q_{\infty} g, h \rangle = \int \langle g, w \rangle \langle h, w
     \rangle d \mu (w), \]
  we obtain
  \[ t \left\langle \int w \otimes (- C)^{- 1} w d \mu (w), e_k \otimes
     e_{\ell} \right\rangle = t \int \langle w, e_k \rangle \langle w, ((-
     C)^{- 1})^{\ast} e_{\ell} \rangle d \mu (w) \]
  \[ = t \langle Q_{\infty} e_k, ((- C)^{- 1})^{\ast} e_{\ell} \rangle = t
     \langle (- C)^{- 1} Q_{\infty} e_k, e_{\ell} \rangle . \]
  By integration by parts we obtain
  \[ (- C)^{- 1} Q_{\infty} = \int_0^{\infty} (- C)^{- 1} e^{C t} Q
     e^{C^{\ast} t} d t \]
  \[ = [(- C)^{- 1} e^{C t} Q e^{C^{\ast} t} ((- C)^{- 1})^{\ast}]_{t = 0}^{t
     = \infty} - \int_0^{\infty} e^{C t} Q e^{C^{\ast} t} ((- C)^{- 1})^{\ast}
     d t \]
  \[ = (- C)^{- 1} Q ((- C)^{- 1})^{\ast} - Q_{\infty} ((- C)^{- 1})^{\ast} \]
  hence the symmetric part reads as
  \[ \tmop{Sym} ((- C)^{- 1} Q_{\infty}) = \frac{1}{2} [(- C)^{- 1} Q_{\infty}
     + Q_{\infty} ((- C)^{- 1})^{\ast}] = \frac{1}{2} (- C)^{- 1} Q ((- C)^{-
     1})^{\ast}, \]
  which cancels the It{\^o}--Stratonovich corrector, as claimed. As a
  consequence, we deduce that
  \[ B^2_{0 t} = \int_0^t B_s \otimes \circ d B_s + t M, \]
  where $M$ is antisymmetric and given by
  \[ \langle M, e_k \otimes e_{\ell} \rangle = \langle \tmop{Ant} ((- C)^{- 1}
     Q_{\infty}) e_k, e_{\ell} \rangle = \left\langle \frac{1}{2} [(- C)^{- 1}
     Q_{\infty} - Q_{\infty} ((- C)^{- 1})^{\ast}] e_k, e_{\ell} \right\rangle
  \]
  or alternatively
  \[ \langle M, e_k \otimes e_{\ell} \rangle = \langle (- C)^{- 1} Q_{\infty}
     e_k, e_{\ell} \rangle - \left\langle \frac{1}{2} (- C)^{- 1} Q ((- C)^{-
     1})^{\ast} e_k, e_{\ell} \right\rangle . \]
  By using Chen's relation again to derive a formula for $B^2_{s t}$ for
  general $s, t \in [0, T]$, \eqref{eq:B2str} follows.

  Step 3: Finally, we shall prove the H\"older continuity of the mapping $\varepsilon\mapsto (Y^{\varepsilon,1}, Y^{\varepsilon,2})$ and the pathwise uniform bound  \eqref{eq:new5}.
  
    Step 3.1: First component of the rough path. For $0 < \varepsilon < \eta
\leqslant 1$ we have
  \[ w^{\varepsilon} (t) - w^{\eta} (t) = (\varepsilon^{- 1 / 2} - \eta^{- 1 /
     2}) \int_0^t e^{\varepsilon^{- 1} C (t - s)} Q^{1 / 2} d W_s \]
  \[ + \eta^{- 1 / 2} \int_0^t (e^{\varepsilon^{- 1} C (t - s)} - e^{\eta^{-
     1} C (t - s)}) Q^{1 / 2} d W_s \backassign I_1 + I_2 . \]
  Using \eqref{eq:6} and \eqref{eq:tr} we obtain
  \[ \mathbb{E} [\| I_1 \|_{H^{\sigma}}^2] = (\varepsilon^{- 1 / 2} - \eta^{-
     1 / 2})^2 \int_0^t \| e^{\varepsilon^{- 1} C (t - s)} Q^{1 / 2} \|_{L_2
     (H ; H^{\sigma})}^2 d s \]
  \[ \lesssim \varepsilon^{- 3} | \varepsilon - \eta |^2 \int_0^t e^{- 2
     \varepsilon^{- 1} \iota (t - s)} d s \lesssim \varepsilon^{- 2} |
     \varepsilon - \eta |^2, \]
  and by \eqref{eq:new2}, \eqref{eq:6} and \eqref{eq:new}  
  \[ \mathbb{E} [\| I_2 \|_{H^{\sigma}}^2] \lesssim \eta^{- 1} \int_0^t \|
     (e^{\varepsilon^{- 1} C (t - s)} - e^{\eta^{- 1} C (t - s)}) Q^{1 / 2}
     \|_{L_2 (H ; H^{\sigma})}^2 d s \]
  \[ \leqslant \eta^{- 1} \int_0^t \| (- C)^{- \vartheta} (e^{(\varepsilon^{- 1}
     - \eta^{- 1}) C (t - s)} - \tmop{Id}) \|_{\mathcal{L} (H^{\sigma} ;
     H^{\sigma})}^{{2}} \| e^{\eta^{- 1} C (t - s)}
     \|_{\mathcal{L} (H^{\sigma} ; H^{\sigma})}^{{2}} \| (-
     C)^{\vartheta} Q^{1 / 2} \|_{L_2 (H ; H^{\sigma})}^{{2}} d s \]
  \[ \lesssim \eta^{- 1} | \varepsilon^{- 1} - \eta^{- 1}
     |^{{\vartheta}} \int_0^t | t - s |^{{\vartheta}} e^{-
{2 \iota \eta^{- 1}} (t - s)} d s \lesssim
     \varepsilon^{{- 2 \vartheta}} | \varepsilon - \eta
     |^{{\vartheta}} . \]
  We deduce
  \begin{equation}
    \mathbb{E} [\| w^{\varepsilon} (t) - w^{\eta} (t) \|_{H^{\sigma}}^2]
    \lesssim \varepsilon^{- 2} | \varepsilon - \eta |^{{\vartheta}}
    . \label{eq:new3}
  \end{equation}
  We proceed with an estimate for $Y^{\varepsilon, 1} - Y^{\eta, 1}$. It holds
  \[ Y^{\varepsilon, 1}_{s t} - Y^{\eta, 1}_{s t} = (\varepsilon^{- 1 / 2} -
     \eta^{- 1 / 2}) \int_s^t w^{\varepsilon}_r d r + \eta^{- 1 / 2} \int_s^t
     (w^{\varepsilon}_r - w^{\eta}_r) d r \]
  hence in view of \eqref{eq:new3} 
  \[ \mathbb{E} [\| Y^{\varepsilon, 1}_{s t} - Y^{\eta, 1}_{s t}
     \|_{H^{\sigma}}^2] \lesssim (\varepsilon^{- 1 / 2} - \eta^{- 1 / 2})^2
     \mathbb{E} \left[ \left( \int_s^t \| w^{\varepsilon}_r \|_{H^{\sigma}} d
     r \right)^2 \right] + \eta^{- 1} \mathbb{E} \left[ \left( \int_s^t \|
     w^{\varepsilon}_r - w^{\eta}_r \|_{H^{\sigma}} d r \right)^2 \right] \]
  \[ \lesssim \varepsilon^{- 3} | \varepsilon - \eta |^2 | t - s |^{{2}} + \eta^{-
     1} \varepsilon^{- 2} | \varepsilon - \eta |^{{\vartheta}} | t -
     s |^{{2}} \]
  \[ \lesssim \varepsilon^{- 3} | \varepsilon - \eta |^{{\vartheta}}
     | t - s |^{{2}} . \]
  By Gaussianity and Kolmogorov's criterion, this leads to
  \begin{equation}
    (\mathbb{E} [\| Y^{\varepsilon, 1} - Y^{\eta, 1} \|_{C^{\beta}_2 ([0, T] ;
    H^{\sigma})}^q])^{1 / q} \lesssim \varepsilon^{{- 3 / 2}} |
    \varepsilon - \eta |^{{\vartheta / 2}} \label{eq:new4}
  \end{equation}
  with $\beta$ as in \eqref{eq:bound}.
  Define $Y^{0, 1} \assign B^1$ so that the first bound in \eqref{eq:bound}
  reads as
  \[ (\mathbb{E} [\| Y^{\varepsilon, 1} - Y^{0, 1} \|_{C^{\beta}_2 ([0, T] ;
     H^{\sigma})}^q])^{1 / q} \lesssim \varepsilon^{(1 - \theta) / 2}, \]
  where $\theta \in (2 / 3, 1]$. This implies the folowing.
  
  First case: if $2 \varepsilon < \eta$ then $\eta < 2 | \varepsilon - \eta |$
  and therefore
  \[ (\mathbb{E} [\| Y^{\varepsilon, 1} - Y^{\eta, 1} \|_{C^{\beta}_2 ([0, T]
     ; H^{\sigma})}^q])^{1 / q} \leqslant (\mathbb{E} [\| Y^{\varepsilon, 1} -
     Y^{0, 1} \|_{C^{\beta}_2 ([0, T] ; H^{\sigma})}^q])^{1 / q} \]
  \[ + (\mathbb{E} [\| Y^{\eta, 1} - Y^{0, 1} \|_{C^{\beta}_2 ([0, T] ;
     H^{\sigma})}^q])^{1 / q} \lesssim \eta^{(1 - \theta) / 2} \lesssim |
     \varepsilon - \eta |^{(1 - \theta) / 2} . \]
 
  Second case: if $\varepsilon < \eta < 2 \varepsilon$ and $\varepsilon
  \leqslant | \varepsilon - \eta |^{\gamma_0}$ for some $\gamma_0$ to be
  chosen, then
  \[ (\mathbb{E} [\| Y^{\varepsilon, 1} - Y^{\eta, 1} \|_{C^{\beta}_2 ([0, T]
     ; H^{\sigma})}^q])^{1 / q} \lesssim \varepsilon^{(1 - \theta) / 2}
     \lesssim | \varepsilon - \eta |^{\gamma_0 (1 - \theta) / 2} . \]
 
  Third case: if $\varepsilon < \eta < 2 \varepsilon$ and $\varepsilon > |
  \varepsilon - \eta |^{\gamma_0}$ then by \eqref{eq:new4}
  \[ (\mathbb{E} [\| Y^{\varepsilon, 1} - Y^{\eta, 1} \|_{C^{\beta}_2 ([0, T]
     ; H^{\sigma})}^q])^{1 / q} \lesssim \varepsilon^{{- 3 / 2}}
     | \varepsilon - \eta |^{{\vartheta / 2}} \lesssim | \varepsilon
     - \eta |^{{\vartheta / 2 - 3 \gamma_0 / 2}} \lesssim |
     \varepsilon - \eta |^{{\vartheta / 4}} \]
  provided ${\gamma_0 = \vartheta / 6}$.
  
  It follows that there exists $\kappa_1 > 0$ so that for every $\varepsilon,
  \eta \in [0, 1]$
  \[ (\mathbb{E} [\| Y^{\varepsilon, 1} - Y^{\eta, 1} \|_{C^{\beta}_2 ([0, T]
     ; H^{\sigma})}^q])^{1 / q} \lesssim {| \varepsilon - \eta
     |^{\kappa_1}}. \]
  By Kolmogorov's criterion we finally deduce the H{\"o}lder continuity of the mapping  
  \[ [0, 1] \rightarrow C^{\beta}_2 ([0, T] ; H^{\sigma}), \qquad \varepsilon
     \mapsto Y^{\varepsilon, 1} \]
     as well as the first bound in
  \eqref{eq:new5}.
  
  Step 3.2: Second component of the rough path. We proceed similarly as before.
  We have
  \[ Y^{\varepsilon, 2}_{s t} = \varepsilon^{- 1} \int_s^t \int_s^r
     w^{\varepsilon}_{\theta} \otimes w^{\varepsilon}_r d \theta d r \]
  and for $0 < \varepsilon < \eta \leqslant 1$
  \[ Y^{\varepsilon, 2}_{s t} - Y^{\eta, 2}_{s t} = (\varepsilon^{- 1} -
     \eta^{- 1}) \int_s^t \int_s^r w^{\varepsilon}_{\theta} \otimes
     w^{\varepsilon}_r d \theta d r \]
  \[ + \eta^{- 1} \int_s^t \int_s^r (w^{\varepsilon}_{\theta} -
     w^{\eta}_{\theta}) \otimes w^{\varepsilon}_r d \theta d r + \eta^{- 1}
     \int_s^t \int_s^r w^{\eta}_{\theta} \otimes (w^{\varepsilon}_r -
     w^{\eta}_r) d \theta d r \backassign I_1 + I_2 + I_3 . \]
  We obtain 
  \[ \mathbb{E} [\| I_1 \|_{H^{\sigma} \otimes H^{\sigma}}^2] \leqslant |
     \varepsilon^{- 1} - \eta^{- 1} |^2 \mathbb{E} \left[ \left( \int_s^t
     \int_s^r \| w^{\varepsilon}_{\theta} \otimes w^{\varepsilon}_r
     \|_{H^{\sigma} \otimes H^{\sigma}} d \theta d r \right)^2 \right]
     \lesssim \varepsilon^{- 4} | \varepsilon - \eta |^2 | t - s |^{{4}} . \]
  Then
  \[ \mathbb{E} [\| I_2 \|_{H^{\sigma} \otimes H^{\sigma}}^2] \lesssim \eta^{-
     1} \mathbb{E} \left[ \left( \int_s^t \| w^{\varepsilon}_{\theta} -
     w^{\eta}_{\theta} \|_{H^{\sigma}} d \theta \int_s^t \| w^{\varepsilon}_r
     \|_{H^{\sigma}} d r \right)^2 \right] ,\]
by Minkowski's integral
  inequality
  \[ \lesssim \eta^{- 1} \left( \int_s^t \int_s^r (\mathbb{E} [\|
     w^{\varepsilon}_{\theta} - w^{\eta}_{\theta} \|^2_{H^{\sigma}} \|
     w^{\varepsilon}_r \|_{H^{\sigma}}^2])^{1 / 2} d \theta d r \right)^2 \]
  \[ \lesssim \eta^{- 1} \left( \int_s^t \int_s^r (\mathbb{E} [\|
     w^{\varepsilon}_{\theta} - w^{\eta}_{\theta} \|^4_{H^{\sigma}}])^{1 / 4}
     (\mathbb{E} [\| w^{\varepsilon}_r \|_{H^{\sigma}}^4])^{1 / 4} d \theta d
     r \right)^2 \]
  and by Gaussianity and \eqref{eq:new3}
  \[ \lesssim {\varepsilon^{- 3} | \varepsilon - \eta |^{\vartheta}
     | t - s |^4} . \]
  The bound for $I_3$ is the same. Altogether, we obtain 
  \[ \mathbb{E} [\| Y^{\varepsilon, 2}_{s t} - Y^{\eta, 2}_{s t}
     \|^2_{H^{\sigma} \otimes H^{\sigma}}] \lesssim
     {\varepsilon^{- {4}} | \varepsilon - \eta |^{\vartheta}} | t - s
     |^{{4}} \]
  hence by Gaussianity and Kolmogorov's criterion
  \[ (\mathbb{E} [\| Y^{\varepsilon, 2}_{s t} - Y^{\eta, 2}_{s t} \|^q_{C^{2
     \beta}_2 ([0, T] ; H^{\sigma} \otimes H^{\sigma})}])^{1 / q} \lesssim
{\varepsilon^{- 2} | \varepsilon - \eta |^{\vartheta/2}} . \]
  Using now the second bound in \eqref{eq:bound} similarly to Step 3.1, we
  deduce that that there exists $\kappa_2 > 0$ so that for every $\varepsilon,
  \eta \in [0, 1]$
  \[ (\mathbb{E} [\| Y^{\varepsilon, 2} - Y^{\eta, 2} \|_{C^{2 \beta}_2 ([0,
     T] ; H^{\sigma})}^q])^{1 / q} \lesssim | \varepsilon - \eta |^{\kappa_2}
  \]
  and consequently the H\"older continuity of the mapping
  \[ [0, 1] \rightarrow C^{2\beta}_2 ([0, T] ; H^{\sigma}), \qquad \varepsilon
     \mapsto Y^{\varepsilon, 2} \]
     and the second bound in \eqref{eq:new5} follow.
\end{proof}

\section{Convergence of the It{\^o}--Stokes drift}\label{s:IS}

\subsection{Uniform estimate of $u^{\varepsilon}$}\label{s:bdu}

Recall that for a fixed $\varepsilon$ all the time integrals involved in
\eqref{eq:u} are classical Lebesgue integrals. Hence after a preliminary step
of Galerkin approximation, the usual energy inequality immediately implies a
uniform bound for $u^{\varepsilon}$. This relies on the cancellation property
of $b$ and yields a.s.
\begin{equation}
  \sup_{t \in [0, T]} \| u^{\varepsilon}_t \|_H^2 + 2 \nu \int_0^T \| \nabla
  u_r^{\varepsilon} \|_H^2 d r \leqslant \| u^{\varepsilon}_0 \|_H^2 .
  \label{eq:estu}
\end{equation}
We note that the inequality comes from the fact that we need to proceed via
Galerkin approximation and use lower semicontinuity, as for 3D Navier--Stokes
equations it is not possible to directly test the equation by the solution
itself.

\subsection{Uniform estimate of $w^{\varepsilon}$}

Using again the change of variables $\tilde{w}_t =
w^{\varepsilon}_{\varepsilon t}$ together with \eqref{est:tilew2}, we have
\begin{equation}
  \sup_{t \in [0, T]} \mathbb{E} [\| w^{\varepsilon}_t \|_{H^{\sigma}}^2] =
  \sup_{t \in [0, T]} \mathbb{E} [\| \tilde{w}_{\varepsilon^{- 1} t}
  \|_{H^{\sigma}}^2] \lesssim 1 \label{eq:estw1}
\end{equation}
and by Gaussianity this also bounds higher moments.

\subsection{Uniform estimate of $r^{\varepsilon}$}\label{s:bdr}

We proceed again via the energy inequality. Testing \eqref{eq:r} by the
solution itself after the preliminary step of Galerkin approximation,
multiplying by $\varepsilon$, taking expectation and applying the cancellation
property of $b$ as well as Young's inequality leads to
\[ \frac{\varepsilon}{2} \partial_t \mathbb{E} [\| r^{\varepsilon} \|_H^2]
   -\mathbb{E} [\langle r^{\varepsilon}, C r^{\varepsilon} \rangle] +
   \varepsilon \mathbb{E} [\| (- A)^{1 / 2} r^{\varepsilon} \|_H^2] \]
\[ \leqslant \varepsilon \mathbb{E} [\langle r^{\varepsilon}, A \varepsilon^{-
   1 / 2} w^{\varepsilon} \rangle] + \varepsilon \mathbb{E} [\langle b
   (u^{\varepsilon}, \varepsilon^{- 1 / 2} w^{\varepsilon}), r^{\varepsilon}
   \rangle] + \varepsilon \mathbb{E} [\langle b (v^{\varepsilon},
   \varepsilon^{- 1 / 2} w^{\varepsilon}), r^{\varepsilon} \rangle] \]
\[ \leqslant \frac{1}{2} \mathbb{E} [\| r^{\varepsilon} \|_H^2] + c\mathbb{E}
   [\| w^{\varepsilon} \|_{H^2}^2] + c\mathbb{E} [\| u^{\varepsilon} \|^4_H]
   {+ c \varepsilon^2}  \mathbb{E} [\| v^{\varepsilon} \|^4_H] + c\mathbb{E}
   \left[ \| w^{\varepsilon} \|_{H^{\theta_0}}^4 \right] . \]
Thus by \eqref{eq:999}
\begin{equation}
  \int_0^T \mathbb{E} [\| r^{\varepsilon} \|_{H^{\gamma}}^2] d t \lesssim
  {\varepsilon} \mathbb{E} [\| r_0^{\varepsilon} \|_H^2] +
  \int_0^T (\mathbb{E} [\| u^{\varepsilon} \|^4_H] + \varepsilon^2 \mathbb{E}
  [\| v^{\varepsilon} \|^4_H] + (\mathbb{E} [\| w^{\varepsilon}
  \|_{H^{\sigma}}^2])^2 + 1) d t. \label{eq:estr}
\end{equation}

While $u^{\varepsilon}$ is already controlled uniformly in $\varepsilon$ and
$\omega$ by \eqref{eq:estu} and $w^{\varepsilon}$ is controlled by
\eqref{eq:estw1}, we still need to establish the necessary bound for
$v^{\varepsilon}$. Again after the preliminary Galerkin approximation, we test
\eqref{eq:v} by $\| v^{\varepsilon} \|_H^{p - 2} v^{\varepsilon}$ to obtain
\[ \frac{1}{p} \partial_t \mathbb{E} [\| v^{\varepsilon} \|_H^p] -
   \varepsilon^{- 1} \mathbb{E} [\| v^{\varepsilon} \|_H^{p - 2} \langle
   v^{\varepsilon}, C v^{\varepsilon} \rangle] +\mathbb{E} [\| v^{\varepsilon}
   \|_H^{p - 2} \| (- A)^{1 / 2} v^{\varepsilon} \|_H^2] \leqslant
   \varepsilon^{- 2} \frac{p - 1}{2} \tmop{Tr} Q\mathbb{E} [\| v^{\varepsilon}
   \|_H^{p - 2}] . \]
Consequently, in view of \eqref{eq:999} it holds for $p = 2$
\[ \varepsilon \int_0^T \mathbb{E} [\| v^{\varepsilon} \|_H^2] d t \leqslant
   \frac{\varepsilon^2}{2} \mathbb{E} [\| v_0^{\varepsilon} \|_H^2] +
   \frac{T}{2} \tmop{Tr} Q \]
and for $p = 4$
\[ \varepsilon^2 \int_0^T \mathbb{E} [\| v^{\varepsilon} \|_H^4] d t \leqslant
   \frac{\varepsilon^3}{2} \mathbb{E} [\| v_0^{\varepsilon} \|_H^4] +
   \frac{3}{2} \tmop{Tr} Q \left( \frac{\varepsilon^2}{2} \mathbb{E} [\|
   v_0^{\varepsilon} \|_H^2] + \frac{T}{2} \tmop{Tr} Q \right) . \]

Plugging this bound as well as \eqref{eq:estu} and \eqref{eq:estw1} into
\eqref{eq:estr} and using the fact that the initial conditions
$r^{\varepsilon}_0 = v^{\varepsilon}_0$ are {such that
$(\varepsilon^{1 / 2} v^{\varepsilon}_0)_{\varepsilon \in (0, 1)}$ is} bounded
in $H$ uniformly in $\varepsilon$, we conclude
\begin{equation}
  \int_0^T \mathbb{E} [\| r^{\varepsilon} \|_{H^{\gamma}}^2] d t \lesssim 1.
  \label{eq:2.5}
\end{equation}
\subsection{Convergence of $r^{\varepsilon}$}\label{s:av}

We shall prove the convergence as $\varepsilon \rightarrow 0$
\begin{equation}
  r^{\varepsilon} \rightarrow \overline{r} \assign \int (- C)^{- 1} b (w, w) d
  \mu (w) \quad \tmop{weakly} \tmop{in} L^2 (0, T ; H^{\gamma}) \tmop{in} L^1
  (\Omega), \label{eq:pp}
\end{equation}
meaning that for every $\psi \in L^2 (0, T)$ and every $\varphi \in H^{-
\gamma}$,
\[ \int_0^T \psi \langle r^{\varepsilon}, \varphi \rangle d t \rightarrow
   \int_0^T \psi \left\langle \overline{r}, \varphi \right\rangle d t \quad
   \tmop{in} L^1 (\Omega) . \]
We also show that for a subsequence the convergence can be strengthened to
\begin{equation}
  r^{\varepsilon} \rightarrow \overline{r} \qquad \tmop{in} H^{- \delta} (0, T
  ; H^{\gamma - \delta})  \text{ a.s.} \label{eq:pp2}
\end{equation}
First, we recall that the uniform bound \eqref{eq:2.5} implies for a
subsequence the weak convergence of $r^{\varepsilon}$ in $L^2 (\Omega \times
[0, T] ; H^{\gamma})$, so for \eqref{eq:pp} we shall identify the limit,
strengthen the convergence in $\omega$ and show that it is not necessary to
pass to a subsequence. To this end, we observe that it follows from
\eqref{eq:r} for all $\psi \in C_c^1 ((0, T))$
\[ \int_0^T \psi r^{\varepsilon} d t = \varepsilon \int_0^T \psi (- C)^{- 1} A
   (\varepsilon^{- 1 / 2} w^{\varepsilon} + r^{\varepsilon}) d t \]
\[ + \varepsilon \int_0^T \psi (- C)^{- 1} b (u^{\varepsilon} + \varepsilon^{-
   1 / 2} w^{\varepsilon} + r^{\varepsilon}, \varepsilon^{- 1 / 2}
   w^{\varepsilon} + r^{\varepsilon}) d t + \varepsilon \int_0^T \partial_t
   \psi (- C)^{- 1} r^{\varepsilon} d t \]
\[ = \int_0^T \psi (- C)^{- 1} b (w^{\varepsilon}, w^{\varepsilon}) d t \]
\[ + \varepsilon^{1 / 2} \int_0^T \psi (- C)^{- 1} [A w^{\varepsilon} + b
   (u^{\varepsilon}, w^{\varepsilon}) + b (w^{\varepsilon}, r^{\varepsilon}) +
   b (r^{\varepsilon}, w^{\varepsilon})] d t \]
\[ + \varepsilon \int_0^T \psi (- C)^{- 1} [A r^{\varepsilon} + b
   (u^{\varepsilon}, r^{\varepsilon}) + b (r^{\varepsilon}, r^{\varepsilon})]
   d t + \varepsilon \int_0^T \partial_t \psi (- C)^{- 1} r^{\varepsilon} d t
   \backassign I^{\varepsilon}_1 + \cdots + I^{\varepsilon}_4 . \]
We claim that as $\varepsilon \rightarrow 0$
\begin{equation}
  I^{\varepsilon}_1 \rightarrow \int_0^T \psi d t \overline{r} \quad \text{in
  } L^2 (\Omega ; H), \label{eq:7}
\end{equation}
whereas $I^{\varepsilon}_2 + I^{\varepsilon}_3 + I^{\varepsilon}_4 \rightarrow
0$ in $H^{- \theta_0}$ in $L^1 (\Omega ; H^{- \theta_0})$.

For the latter convergence, we estimate by \eqref{eq:89}
\[ \mathbb{E} \left[ \| I^{\varepsilon}_2 \|_{H^{- \theta_0}} \right] \lesssim
   \varepsilon^{1 / 2} \mathbb{E} \left[ \int_0^T \| A w^{\varepsilon} + b
   (u^{\varepsilon}, w^{\varepsilon}) + b (w^{\varepsilon}, r^{\varepsilon}) +
   b (r^{\varepsilon}, w^{\varepsilon}) \|_{H^{- \theta_0}} d t \right] \]
\[ \lesssim \varepsilon^{1 / 2} \mathbb{E} \left[ \int_0^T (\| w^{\varepsilon}
   \|_H + \| u^{\varepsilon} \|_H^2 + \| w^{\varepsilon} \|_H^2 + \|
   r^{\varepsilon} \|_H^2) d t \right] \lesssim \varepsilon^{1 / 2}, \]
where the implicit constant does not depend on $\varepsilon$ and the last
inequality follows from \eqref{eq:estu}, \eqref{eq:estw1}, \eqref{eq:2.5}.
Similarly,
\[ \mathbb{E} \left[ \| I^{\varepsilon}_3 \|_{H^{- \theta_0}} + \|
   I^{\varepsilon}_4 \|_{H^{- \theta_0}} \right] \lesssim \varepsilon
   \mathbb{E} \left[ \int_0^T (\| r^{\varepsilon} \|_H + \| u^{\varepsilon}
   \|_H^2 + \| r^{\varepsilon} \|_H^2) d t \right] \lesssim \varepsilon . \]
In order to establish \eqref{eq:7}, we make use of the ergodic theorem from
Corollary \ref{p:erg}. More precisely, we rewrite
\begin{equation}
  I^{\varepsilon}_1 = \int_0^T \psi (- C)^{- 1} b (w^{\varepsilon},
  w^{\varepsilon}) d t = - \int_0^T \partial_t \psi \int_0^t (- C)^{- 1} b
  (w_s^{\varepsilon}, w_s^{\varepsilon}) d s d t \label{eq:mmm}
\end{equation}
and observe that $\tilde{w}_t = w^{\varepsilon}_{\varepsilon t}$ leads to
\[ \int_0^t (- C)^{- 1} b (w_s^{\varepsilon}, w_s^{\varepsilon}) d s =
   \varepsilon \int_0^{\varepsilon^{- 1} t} (- C)^{- 1} b (\tilde{w}_s,
   \tilde{w}_s) d s. \]
Letting $F (w) \assign (- C)^{- 1} b (w, w)$ we see that by \eqref{eq:89} $F$
satisfies the assumption of Corollary~\ref{p:erg}. Hence we deduce for every
$t \in [0, T]$
\[ \lim_{\varepsilon \rightarrow 0} \int_0^t (- C)^{- 1} b (w_s^{\varepsilon},
   w_s^{\varepsilon}) d s = t \int (- C)^{- 1} b (w, w) d \mu (w)  \quad
   \tmop{in} L^2 (\Omega ; H) . \]
Plugging this into \eqref{eq:mmm} we use \eqref{eq:estw1} and apply dominated
convergence theorem to obtain \eqref{eq:7}. Finally, using this together with
a density argument and the uniform bound \eqref{eq:2.5}, we obtain
\eqref{eq:pp}.

We note that by weak--strong convergence this is enough to pass to the limit
in the It{\^o}--Stokes drift in \eqref{eq:u}, namely in the term $b
(r^{\varepsilon}, u^{\varepsilon}) d t$, provided e.g. the sequence
$u^{\varepsilon}$ converges strongly in $L^2 (0, T ; H)$ a.s., which we prove
in Section~\ref{s:tight} below using the stochastic compactness method.

It remains to verify \eqref{eq:pp2}, i.e. to show that the convergence can be
strengthened to an a.s. convergence but in a slightly worse topology. The
uniform bound \eqref{eq:2.5} implies in particular tightness of
$r^{\varepsilon}$ in $H^{- \delta} (0, T ; H^{\gamma - \delta})$ for $\delta >
0$. Consequently, up to a subsequence, by Skorokhod representation theorem,
$r^{\varepsilon}$ converges in law in $H^{- \delta} (0, T ; H^{\gamma -
\delta})$. Since the limit, i.e. $\overline{r}$, is deterministic, the
convergence $r^{\varepsilon} \rightarrow \overline{r}$ holds true in in $H^{-
\delta} (0, T ; H^{\gamma - \delta})$ in probability. Hence, taking a further
subsequence, $r^{\varepsilon} \rightarrow \overline{r}$ in $H^{- \delta} (0, T
; H^{\gamma - \delta})$ a.s.

\section{Tightness and passage to the limit}\label{s:tight1}

This section is devoted to the completion of the proof of
Theorem~\ref{thm:main}, which proceeds in several steps. First, it is
necessary to establish further uniform estimates in the rough path formulation
\eqref{eq:urp}. Namely, we shall prove a uniform estimate for the remainders
$u^{\varepsilon, \natural}$ as well as a uniform time regularity of
$u^{\varepsilon}$. We present these results in Section~\ref{s:rem} and
Section~\ref{s:timereg} below. We note that since $r^{\varepsilon}$ is only
controlled in expectation, cf. \eqref{eq:2.5}, also the bounds from
Section~\ref{s:rem} and Section~\ref{s:timereg} are only uniform after taking
expectations. This makes the final passage to the limit argument in
Section~\ref{s:tight} delicate.

\begin{remark}
  \label{r:11}
  For future use, we recall that by \eqref{eq:new5}
  and the estimates in Section~\ref{s:rr}, we have a pathwise control of the
  required bounds of $\mathbb{A}^{\varepsilon}$, namely \eqref{eq:a1a1},
  \eqref{eq:a2a2} and \eqref{eq:a2a22} and they hold uniformly in
  $\varepsilon$. More precisely, there exists $K \in L^q (\Omega)$ for all $q
  \in [1, \infty)$ such that $\tmmathbf{P}$-a.s.
  \[ \| \mathbb{A}_{s t}^{\varepsilon, 1} \|^{1 / \alpha}_{\mathcal{L} (H^{-
     n}, H^{- (n + 1)})} + \| \mathbb{A}_{s t}^{\varepsilon, 2} \|^{1 / (2
     \alpha)}_{\mathcal{L} (H^{- n}, H^{- (n + 2)})} + \| \mathbb{A}_{s
     t}^{\varepsilon, 2, \ast} \|^{1 / (2 \alpha)}_{\mathcal{L} (H^{\theta_0 +
     1}, L^{\infty})} \]
  \begin{equation}
    \lesssim (\| Y^{\varepsilon, 1} \|_{C^{\alpha}_2 ([0, T] ; H^{\sigma})}^{1
    / \alpha} + \| Y^{\varepsilon, 2} \|_{C^{2 \alpha}_2 ([0, T] ; H^{\sigma}
    \otimes H^{\sigma})}^{1 / (2 \alpha)}) | t - s | \leqslant K | t - s |
    \backassign \omega_{\mathbb{A}} (s, t) . \label{eq:contr}
  \end{equation}
\end{remark}

\subsection{Uniform estimate of $u^{\varepsilon, \natural}$}\label{s:rem}

For a fixed $\varepsilon$, the solution $u^{\varepsilon}$ satisfies the rough
path form of \eqref{eq:u}, namely \eqref{eq:urp} in the sense of
Definition~\ref{d:sol}. We know a priori that the remainder $u^{\varepsilon,
\natural}$ belongs to $C^{p / 3 - \tmop{var}}_2 ([0, T] ; H^{- 3})$. As a
matter of fact, it is even much better as there is no rough-in-time term in
\eqref{eq:u}. But this regularity depends on $\varepsilon$ and we seek an
estimate on $u^{\varepsilon, \natural}$ which is uniform in $\varepsilon$. The
equation \eqref{eq:urp} is very similar to the rough path formulation of the
3D Navier--Stokes equations perturbed by a transport noise in
{\cite{MR3918521}}, see Section 2.5 and particularly (2.18) in
{\cite{MR3918521}}. The difference is that the drift part $\mu^{\varepsilon}$
in our case includes the additional It{\^o}--Stokes drift $\mu^{\varepsilon,
2}$, while $\mu^{\varepsilon, 1}$ coincides with the drift in
{\cite{MR3918521}}.

In Section~3 in {\cite{MR3918521}} and particularly in Lemma~3.1 and
Lemma~3.3, two bounds on $\mu^{\varepsilon, 1}$ were used and they are
satisfied in our case as well, namely,
\begin{equation*}
  \| \delta \mu^{\varepsilon, 1}_{s t} \|_{H^{- 1}} \lesssim \int_s^t (1 + \|
  u^{\varepsilon}_{\tau} \|_{H^1})^2 d \tau \label{eq:43}
\end{equation*}
and
\begin{equation*}
  \| \delta \mu^{\varepsilon, 1}_{s t} \|_{H^{- \theta_0}} \lesssim | t - s |
  (1 + \| u^{\varepsilon} \|_{L^{\infty}_T H})^2 . \label{eq:44}
\end{equation*}
In the case of $\mu^{\varepsilon, 2}$, we estimate for a fixed $\varepsilon$
\[ | \delta \mu^{\varepsilon, 2}_{s t} (\varphi) | = \left| \int_s^t \langle b
   (r^{\varepsilon}_{\tau}, u^{\varepsilon}_{\tau}), \varphi \rangle d \tau
   \right| \lesssim \| \varphi \|_{L^{\infty}} \int_s^t \|
   r^{\varepsilon}_{\tau} \|_H \| u^{\varepsilon}_{\tau} \|_{H^1} d \tau \]
\begin{equation*}
  \lesssim \| \varphi \|_{L^{\infty}} \left( \int_s^t \|
  r^{\varepsilon}_{\tau} \|_H^2 d \tau \right)^{1 / 2} \left( \int_s^t \|
  u^{\varepsilon}_{\tau} \|^2_{H^1} d \tau \right)^{1 / 2} . \label{eq:mur1}
\end{equation*}
While a.s. the right hand side defines a control and hence $\mu^{\varepsilon,
2}$ is of finite variation e.g. in $H^{- \theta_0 + 1}$ by Sobolev embedding,
this control is not uniform in $\varepsilon$. Indeed, in view of
\eqref{eq:2.5} a uniform in $\varepsilon$ bound can only be obtained in
expectation and the a.s. bound following from \eqref{eq:pp2} is in a worse
space. Hence we prefer to use \eqref{eq:2.5} in what follows. This is a
striking difference from the setting of {\cite{MR3918521}}. Another difference
lies in the weaker spatial regularity of $\mu^{\varepsilon, 2}$ compared to
$\mu^{\varepsilon, 1}$, namely, $H^{- \theta_0 + 1}$ instead of $H^{- 1}$. As
a consequence of this fact, we require the additional conditions on the
unbounded rough driver $(\mathbb{A}^{\varepsilon, 1}, \mathbb{A}^{\varepsilon,
2})$, cf. \eqref{eq:a1a11}, \eqref{eq:a2a22} compared to (2.10) in
{\cite{MR3918521}}.

We denote the control associated to the $C^{1 - \tmop{var}} ([0, T] ; H^{-
\theta_0 + 1})$ norm of $\mu^{\varepsilon}$ by
\begin{equation}
  \omega_{\mu^{\varepsilon}} (s, t) = \int_s^t (1 + \| u^{\varepsilon}_{\tau}
  \|_{H^1})^2 d \tau + \left( \int_s^t \| r^{\varepsilon}_{\tau} \|_H^2 d \tau
  \right)^{1 / 2} \left( \int_s^t \| u^{\varepsilon}_{\tau} \|^2_{H^1} d \tau
  \right)^{1 / 2} . \label{eq:26a}
\end{equation}
With this in hand, we repeat the proof of Lemma~3.1 in {\cite{MR3918521}} in
order to obtain the counterpart of (3.3). Applying the increment operator
$\delta$ to \eqref{eq:urp} for $s \leqslant \theta \leqslant t$ and using the
Chen's relation yields
\[ \delta u^{\varepsilon, \natural}_{s \theta t} (\varphi) = \delta
   u^{\varepsilon}_{s \theta} (\mathbb{A}^{\varepsilon, 2, \ast}_{\theta t}
   \varphi) + u^{\varepsilon, \sharp}_{s \theta} (\mathbb{A}^{\varepsilon, 1,
   \ast}_{\theta t} \varphi), \]
where
\begin{equation}
  u^{\varepsilon, \sharp}_{s \theta} = \delta u^{\varepsilon}_{s \theta}
  -\mathbb{A}^{\varepsilon, 1}_{s \theta} u_s^{\varepsilon} = \delta
  \mu^{\varepsilon}_{s \theta} +\mathbb{A}^{\varepsilon, 2}_{s \theta}
  u^{\varepsilon}_s + u^{\varepsilon, \natural}_{s \theta} . \label{eq:sharp}
\end{equation}
To exploit the interplay between time and space regularity, we let
$(J^{\eta})_{\eta \in (0, 1]}$ be a family of smoothing operators as in
{\cite{MR3918521}} and we split
\begin{equation}
  \delta u^{\varepsilon, \natural}_{s \theta t} (\varphi) = \delta
  u^{\varepsilon, \natural}_{s \theta t} ((\tmop{Id} - J^{\eta}) \varphi) +
  \delta u^{\varepsilon, \natural}_{s \theta t} (J^{\eta} \varphi) .
  \label{eq:split}
\end{equation}
In the estimate of the first term in \eqref{eq:split}, we make use of the
middle expression in \eqref{eq:sharp}, i.e. for $\varphi \in H^3$, $\| \varphi
\|_{H^3} \leqslant 1$
\[ | \delta u^{\varepsilon, \natural}_{s \theta t} ((\tmop{Id} - J^{\eta})
   \varphi) | \leqslant | \delta u^{\varepsilon}_{s \theta}
   (\mathbb{A}^{\varepsilon, 2, \ast}_{\theta t} (\tmop{Id} - J^{\eta})
   \varphi) | + | (\delta u^{\varepsilon}_{s \theta} -\mathbb{A}^{\varepsilon,
   1}_{s \theta} u_s^{\varepsilon}) (\mathbb{A}^{\varepsilon, 1, \ast}_{\theta
   t} (\tmop{Id} - J^{\eta}) \varphi) | \]
\[ \leqslant \| u^{\varepsilon} \|_{L_T^{\infty} H} (\|
   \mathbb{A}^{\varepsilon, 2, \ast}_{\theta t} (\tmop{Id} - J^{\eta}) \varphi
   \|_H + \| \mathbb{A}^{\varepsilon, 1, \ast}_{\theta t} (\tmop{Id} -
   J^{\eta}) \varphi \|_H + \| \mathbb{A}^{\varepsilon, 1, \ast}_{s \theta}
   \mathbb{A}^{\varepsilon, 1, \ast}_{\theta t} (\tmop{Id} - J^{\eta}) \varphi
   \|_H) . \]
Recalling the notation \eqref{eq:contr}, the mollification estimates (2.7) in
{\cite{MR3918521}} and letting $p = 1 / \alpha$, we obtain
\[ | \delta u^{\varepsilon, \natural}_{s \theta t} ((\tmop{Id} - J^{\eta})
   \varphi) | \lesssim \| u^{\varepsilon} \|_{L_T^{\infty} H}
   (\omega_{\mathbb{A}} (s, t)^{2 / p} \| (\tmop{Id} - J^{\eta}) \varphi
   \|_{H^2} + \omega_{\mathbb{A}} (s, t)^{1 / p} \| (\tmop{Id} - J^{\eta})
   \varphi \|_{H^1}) \]
\begin{equation}
  \lesssim \| u^{\varepsilon} \|_{L_T^{\infty} H} (\omega_{\mathbb{A}} (s,
  t)^{2 / p} \eta + \omega_{\mathbb{A}} (s, t)^{1 / p} \eta^2) .
  \label{eq:ppp}
\end{equation}
In order to estimate the second term in \eqref{eq:split}, we apply the
expression on the very right hand side of \eqref{eq:sharp} for
$u^{\varepsilon, \sharp}$ and we also use the equation \eqref{eq:urp} for
$\delta u^{\varepsilon}$. This gives
\[ \delta u^{\varepsilon, \natural}_{s \theta t} (J^{\eta} \varphi) = \delta
   \mu^{\varepsilon}_{s \theta} (\mathbb{A}^{\varepsilon, 2, \ast}_{\theta t}
   J^{\eta} \varphi) + u^{\varepsilon}_s (\mathbb{A}_{s \theta}^{\varepsilon,
   1, \ast} \mathbb{A}^{\varepsilon, 2, \ast}_{\theta t} J^{\eta} \varphi) +
   u^{\varepsilon}_s (\mathbb{A}_{s \theta}^{\varepsilon, 2, \ast}
   \mathbb{A}^{\varepsilon, 2, \ast}_{\theta t} J^{\eta} \varphi) +
   u^{\varepsilon, \natural}_{s \theta} (\mathbb{A}^{\varepsilon, 2,
   \ast}_{\theta t} J^{\eta} \varphi) \]
\[ + \delta \mu^{\varepsilon}_{s \theta} (\mathbb{A}^{\varepsilon, 1,
   \ast}_{\theta t} \varphi) + u^{\varepsilon}_s (\mathbb{A}^{\varepsilon, 2,
   \ast}_{s \theta} \mathbb{A}^{\varepsilon, 1, \ast}_{\theta t} \varphi) +
   u^{\varepsilon, \natural}_{s \theta} (\mathbb{A}^{\varepsilon, 1,
   \ast}_{\theta t} \varphi) . \]
Let us first focus on the terms with $\mu^{\varepsilon, 1}$, since they did
not appear in {\cite{MR3918521}}. In particular, we estimate using
\eqref{eq:a1a11}
\[ | \delta \mu^{\varepsilon, 1}_{s \theta} (\mathbb{A}_{\theta
   t}^{\varepsilon, 1, \ast} J^{\eta} \varphi) | \leqslant
   \omega_{\mu^{\varepsilon}} (s, t) \| \mathbb{A}_{\theta t}^{\varepsilon, 1,
   \ast} J^{\eta} \varphi \|_{L^{\infty}} \lesssim \omega_{\mu^{\varepsilon}}
   (s, t) \omega_{\mathbb{A}} (s, t)^{1 / p} \| \varphi \|_{H^3} \lesssim
   \omega_{\mu^{\varepsilon}} (s, t) \omega_{\mathbb{A}} (s, t)^{1 / p}, \]
and similarly using \eqref{eq:a2a22}
\[ | \delta \mu^{\varepsilon, 1}_{s \theta} (\mathbb{A}_{\theta
   t}^{\varepsilon, 2, \ast} J^{\eta} \varphi) | \leqslant
   \omega_{\mu^{\varepsilon}} (s, t) \| \mathbb{A}_{\theta t}^{\varepsilon, 2,
   \ast} J^{\eta} \varphi \|_{L^{\infty}} \]
\[ \lesssim \omega_{\mu^{\varepsilon}} (s, t) \omega_{\mathbb{A}} (s, t)^{2 /
   p} \| J^{\eta} \varphi \|_{H^{\theta_0 + 1}} \lesssim
   \omega_{\mu^{\varepsilon}} (s, t) \omega_{\mathbb{A}} (s, t)^{2 / p}
   \eta^{2 - \theta_0} . \]
The other terms can be estimated as in {\cite{MR3918521}}.

Recall that we aim at showing a bound of $u^{\varepsilon, \natural}$ in $C^{p
/ 3 - \tmop{var}}_{2, \tmop{loc}} ([0, T] ; H^{- 3})$, that is, we intend to
bound the control
\[ \omega_{\varepsilon, \natural} (s, t) \assign \| u^{\varepsilon, \natural}
   \|^{p / 3}_{C_2^{p / 3 - \tmop{var}} ([s, t] ; H^{- 3})} . \]
With this notation, we obtain
\[ | \delta u^{\varepsilon, \natural}_{s \theta t} (J^{\eta} \varphi) |
   \leqslant \omega_{\mu^{\varepsilon}} (s, t) \omega_{\mathbb{A}} (s, t)^{2 /
   p} (1 + \eta^{2 - \theta_0}) + \| u^{\varepsilon} \|_{L_T^{\infty} H}
   \omega_{\mathbb{A}} (s, t)^{3 / p} \]
\[ + \| u^{\varepsilon} \|_{L_T^{\infty} H} \omega_{\mathbb{A}} (s, t)^{4 / p}
   \eta^{- 1} + \omega_{\varepsilon, \natural} (s, t)^{3 / p}
   \omega_{\mathbb{A}} (s, t)^{2 / p} \eta^{- 2} \]
\[ + \omega_{\mu^{\varepsilon}} (s, t) \omega_{\mathbb{A}} (s, t)^{1 / p} + \|
   u^{\varepsilon} \|_{L_T^{\infty} H} \omega_{\mathbb{A}} (s, t)^{3 / p} +
   \omega_{\varepsilon, \natural} (s, t)^{3 / p} \omega_{\mathbb{A}} (s, t)^{1
   / p} \eta^{- 1} . \]
With the choice of $\eta = \omega_{\mathbb{A}} (s, t)^{1 / p} \lambda$ for
some constant $\lambda > 0$ to be chosen below, the above bound reads as
\[ | \delta u^{\varepsilon, \natural}_{s \theta t} (J^{\eta} \varphi) |
   \lesssim \omega_{\mu^{\varepsilon}} (s, t) (\omega_{\mathbb{A}} (s, t)^{2 /
   p} + \omega_{\mathbb{A}} (s, t)^{(4 - \theta_0) / p} \lambda^{2 - \theta_0}
   + \omega_{\mathbb{A}} (s, t)^{1 / p}) \]
\[ + \| u^{\varepsilon} \|_{L_T^{\infty} H} \omega_{\mathbb{A}} (s, t)^{3 / p}
   (1 + \lambda^{- 1}) + \omega_{\varepsilon, \natural} (s, t)^{3 / p}
   (\lambda^{- 2} + \lambda^{- 1}), \]
where the implicit constant is universal and independent of $\varepsilon, s,
\theta, t$ or other data of the equations. Combining this with \eqref{eq:ppp}
we obtain
\[ \| \delta u^{\varepsilon, \natural}_{s \theta t} \|_{H^{- 3}} \lesssim
   \omega_{\mu^{\varepsilon}} (s, t) (\omega_{\mathbb{A}} (s, t)^{2 / p} +
   \omega_{\mathbb{A}} (s, t)^{(4 - \theta_0) / p} \lambda^{2 - \theta_0} +
   \omega_{\mathbb{A}} (s, t)^{1 / p}) \]
\[ + \| u^{\varepsilon} \|_{L^{\infty} H} \omega_{\mathbb{A}} (s, t)^{3 / p}
   (1 + \lambda^{- 1} + \lambda + \lambda^2) + \omega_{\varepsilon, \natural}
   (s, t)^{3 / p} (\lambda^{- 2} + \lambda^{- 1}) . \]

To get the desired estimate of $u^{\varepsilon, \natural}$, we require that
all the terms on the right hand side are written as $\omega (s, t)^{3 / p}$
for some control $\omega$. Recall that for every two controls $\omega_1$ and
$\omega_2$, $\omega_1^a \omega_2^b$ is a control provided $a + b \geqslant 1$.
Thus particularly for the first term on the right hand side, we need $1 + (4 -
\theta_0) / p \geqslant 3 / p$. If $\theta_0 = 1 + 3 / 2 + \delta$ with
$\delta > 0$ small, this boils down to $1 / p \leqslant 2 / (3 + 2 \delta)$
which is indeed possible for every $p \in (2, 3)$ by choosing $\delta$
sufficiently small.

Applying the sewing lemma, Lemma B.1 in {\cite{MR3918521}}, we deduce
\[ \| u^{\varepsilon, \natural}_{s t} \|_{H^{- 3}} \lesssim
   \omega_{\mu^{\varepsilon}} (s, t) (\omega_{\mathbb{A}} (s, t)^{2 / p} +
   \omega_{\mathbb{A}} (s, t)^{(4 - \theta_0) / p} \lambda^{2 - \theta_0} +
   \omega_{\mathbb{A}} (s, t)^{1 / p}) \]
\[ + \| u^{\varepsilon} \|_{L_T^{\infty} H} \omega_{\mathbb{A}} (s, t)^{3 / p}
   (1 + \lambda^{- 1} + \lambda + \lambda^2) + \omega_{\varepsilon, \natural}
   (s, t)^{3 / p} (\lambda^{- 2} + \lambda^{- 1}) \]
and accordingly
\[ \omega_{\varepsilon, \natural} (s, t) \leqslant c
   [\omega_{\mu^{\varepsilon}} (s, t)^{p / 3} (\omega_{\mathbb{A}} (s, t)^{2 /
   3} + \omega_{\mathbb{A}} (s, t)^{(4 - \theta_0) / 3} \lambda^{p (2 -
   \theta_0) / 3} + \omega_{\mathbb{A}} (s, t)^{1 / 3}) \nobracket \]
\[ + \| u^{\varepsilon} \|^{p / 3}_{L_T^{\infty} H} \omega_{\mathbb{A}} (s, t)
   (1 + \lambda^{- 1} + \lambda + \lambda^2)^{p / 3} + \nobracket
   \omega_{\varepsilon, \natural} (s, t) (\lambda^{- 2} + \lambda^{- 1})^{p /
   3}] . \]
To close the estimate, we choose $\lambda$ such that $c (\lambda^{- 2} +
\lambda^{- 1})^{p / 3} = 1 / 2$. Then we can absorb $\omega_{\varepsilon,
\natural}$ from the right hand side into the left hand side. In addition, for
the mollifier estimates, we need to guarantee $\eta \in (0, 1]$. In view of
the definition of $\eta$ above and $\omega_{\mathbb{A}}$ and $K$ in
\eqref{eq:contr} we have
\[ \eta = (K | t - s |)^{1 / p} \lambda . \]
Hence with the above choice of (deterministic) $\lambda$, there exists $L = L
(\omega)$ so that $\eta \leqslant 1$ provided $| t - s | \leqslant L
(\omega)$. In particular, we may choose $L = \frac{1}{K \lambda^p} \sim
\frac{1}{K}$.

This leads us to the following analog of (3.3) in {\cite{MR3918521}}
\begin{equation}
  \omega_{\varepsilon, \natural} (s, t) \lesssim \omega_{\mu^{\varepsilon}}
  (s, t)^{p / 3} (\omega_{\mathbb{A}} (s, t)^{2 / 3} + \omega_{\mathbb{A}} (s,
  t)^{(4 - \theta_0) / 3} + \omega_{\mathbb{A}} (s, t)^{1 / 3}) + \|
  u^{\varepsilon} \|_{L^{\infty}_T H}^{p / 3} \omega_{\mathbb{A}} (s, t),
  \label{eq:3.3}
\end{equation}
which holds true provided $| t - s |$ was sufficiently small, given by the
random but $\varepsilon$-independent upper bound $L$. This is also the reason
why the remainder only belongs to the local space $C^{p / 3 - \tmop{var}}_{2,
\tmop{loc}} ([0, T] ; H^{- 3})$.

Finally, we note that for a given constant $a > 0$, it holds
$\omega_{\varepsilon, \natural} (s, t) \leqslant a$ provided $| t - s |
\leqslant \tilde{L}^{\varepsilon}$ where $\tilde{L}^{\varepsilon}$ is random
and chosen in dependence of $a, \omega_{\mu^{\varepsilon}} (0, T), \|
u^{\varepsilon} \|_{L_T^{\infty} H}$ and $K$ based on \eqref{eq:3.3} and
\eqref{eq:contr}. For instance, we may define
\begin{equation}
  \tilde{L}^{\varepsilon} \assign \frac{1}{K} \min \left\{ \left( \frac{a}{4
  \omega_{\mu^{\varepsilon}} (0, T)^{p / 3}} \right)^{3 / 2}, \left(
  \frac{a}{4 \omega_{\mu^{\varepsilon}} (0, T)^{p / 3}} \right)^{3 / (4 -
  \theta_0)}, \left( \frac{a}{4 \omega_{\mu^{\varepsilon}} (0, T)^{p / 3}}
  \right)^3, \frac{a}{4 \| u^{\varepsilon} \|^{p / 3}_{L_T^{\infty} H}}
  \right\} . \label{eq:L}
\end{equation}
\subsection{Uniform time regularity of $u^{\varepsilon}$}\label{s:timereg}

As the next step, we proceed with an analog of Lemma~3.3 in {\cite{MR3918521}}
proving time regularity of $u^{\varepsilon}$. Due to the worse spatial
regularity of $\mu^{\varepsilon, 2}$ we estimate only
$\omega_{u^{\varepsilon}} (s, t) \assign \| u \|^p_{C^{p - \tmop{var}} ([s, t]
; H^{- \theta_0 + 1})}$ as
\begin{equation}
  \omega_{u^{\varepsilon}} (s, t) \lesssim (1 + \| u^{\varepsilon}
  \|_{L^{\infty}_T H})^p (\omega_{\varepsilon, \natural} (s, t) +
  \omega_{\mu^{\varepsilon}} (s, t)^p + \omega_{\mathbb{A}} (s, t)),
  \label{eq:80}
\end{equation}
which holds true provided $| t - s | \leqslant \tilde{L}^{\varepsilon}
(\omega)$ for a random bound $\tilde{L}^{\varepsilon}$ given by (cf.
\eqref{eq:L})
\[ \tilde{L}^{\varepsilon} \assign \frac{1}{K} \min \left\{ \left( \frac{1 /
   2^p}{4 \omega_{\mu^{\varepsilon}} (0, T)^{p / 3}} \right)^{3 / 2}, \left(
   \frac{1 / 2^p}{4 \omega_{\mu^{\varepsilon}} (0, T)^{p / 3}} \right)^{3 / (4
   - \theta_0)}, \left( \frac{1 / 2^p}{4 \omega_{\mu^{\varepsilon}} (0, T)^{p
   / 3}} \right)^3, \frac{1 / 2^p}{4 \| u^{\varepsilon} \|^{p /
   3}_{L_T^{\infty} H}}, 1 / 2^p \right\} . \]
Within this estimate we made use of the bound for $\varphi \in H^{\theta_0 -
1}$, $\| \varphi \|_{H^{\theta_0}} \leqslant 1$
\[ | \delta \mu^{\varepsilon}_{s t} (J^{\eta} \varphi) | \leqslant
   \omega_{\mu^{\varepsilon}} (s, t) \| J^{\eta} \varphi \|_{H^{\theta_0 - 1}}
   \lesssim \omega_{\mu^{\varepsilon}} (s, t) \]
and we chose $\eta = \omega_{\varepsilon, \natural} (s, t)^{1 / p} +
\omega_{\mathbb{A}} (s, t)^{1 / p}$ as in Lemma~3.3 in {\cite{MR3918521}}. The
requirement $\eta \in (0, 1]$ is the reason for the upper bound
$\tilde{L}^{\varepsilon} (\omega)$.

Accordingly, we can bound
\[ \| u^{\varepsilon} \|^p_{C^{p - \tmop{var}} ([0, T] ; H^{- \theta_0 + 1})}
   = \sup_{\pi \in \mathcal{P} ([0, T])} \sum_{(s, t) \in \pi} \| \delta
   u^{\varepsilon}_{s t} \|_{H^{- \theta_0 + 1}}^p, \]
where the supremum is taken over all partitions $\pi$ of the interval $[0,
T]$. Here, we can apply \eqref{eq:80} whenever the mesh size of a partition
$\pi$ is at most $\tilde{L}^{\varepsilon}$. If the mesh size of $\pi$ is
bigger, we can refine it so that the sum of norms of increments is bounded by
a sum over the finer partition times an implicit constant which is at most $(T
/ \tilde{L}^{\varepsilon})^p$. We deduce
\[ \| u^{\varepsilon} \|^p_{C^{p - \tmop{var}} ([0, T] ; H^{- \theta_0 + 1})}
   \lesssim \left( \frac{T}{\tilde{L}^{\varepsilon}} \right)^p (1 + \|
   u^{\varepsilon} \|_{L_T^{\infty} H})^p (\omega_{\varepsilon, \natural} (0,
   T) + \omega_{\mu^{\varepsilon}} (0, T)^p +
   \omega_{\mathbb{A}} (0, T)), \]
and we used the fact that controls are super-additive and for a control
$\omega$ and $a \geqslant 1$, also $\omega^a$ is a control, i.e.
super-additive.

In view of \eqref{eq:26a}, \eqref{eq:3.3}, \eqref{eq:estu} and \eqref{eq:2.5}
and using the form of $\tilde{L}^{\varepsilon}$ we therefore obtain for some
$\kappa > 0$
\[ \mathbb{E} \left[ \| u^{\varepsilon} \|^{\kappa}_{C^{p - \tmop{var}} ([0,
   T] ; H^{- \theta_0 + 1})} \right] \lesssim_T {(1 + \| u^{\varepsilon}_0
   \|_H)^{\kappa}}  \]
\[ \times \mathbb{E} [K^{\kappa} (\omega_{\mu^{\varepsilon}} (0, T)^{\kappa p
   / 2} + \omega_{\mu^{\varepsilon}} (0, T)^{\kappa p / (4 - \theta_0)} +
   \omega_{\mu^{\varepsilon}} (0, T)^{\kappa p} + \| u^{\varepsilon}
   \|^{\kappa p / 3}_{L_T^{\infty} H}) \nobracket \]
\[ \nobracket \times (\omega_{\varepsilon, \natural} (0, T)^{\kappa / p} +
   \omega_{\mu^{\varepsilon}} (0, T)^{\kappa} + \omega_{\mathbb{A}} (0,
   T)^{\kappa / p})] . \]
The reason we included $\kappa$ is to guarantee the integrability of
$\omega_{\mu^{\varepsilon}}$ which due to the limited bound of
$r^{\varepsilon}$ only admits the second moment (cf. \eqref{eq:2.5} and
\eqref{eq:26a}). Choosing $\kappa$ small enough, recalling that $K$ and
$\omega_{\mathbb{A}}$ admit moments of all orders by Remark~\ref{r:11}, we
further estimate as
\begin{equation}
  \mathbb{E} \left[ \| u^{\varepsilon} \|^{\kappa}_{C^{p - \tmop{var}} ([0, T]
  ; H^{- \theta_0 + 1})} \right] \lesssim 1 +\mathbb{E}
  [\omega_{\mu^{\varepsilon}} (0, T)^2] \lesssim 1, \label{eq:timeregu}
\end{equation}
where we bounded uniformly in $\varepsilon$ as follows
\[ \mathbb{E} [\omega_{\mu^{\varepsilon}} (0, T)^2] \lesssim \mathbb{E} \left[
   \left( \int_0^T (1 + \| u^{\varepsilon}_{\tau} \|_{H^1})^2 d \tau \right)^2
   + \int_0^T \| r^{\varepsilon}_{\tau} \|_H^2 d \tau \int_0^T \|
   u^{\varepsilon}_{\tau} \|^2_{H^1} d \tau \right] \]
\[ \lesssim 1 + \| u^{\varepsilon}_0 \|_H^4 + \| u^{\varepsilon}_0 \|_H^2
   \mathbb{E} \left[ \int_0^T \| r^{\varepsilon}_{\tau} \|_H^2 d \tau \right]
   \lesssim 1. \]
This shows a uniform bound of $u^{\varepsilon}$ in $C^{p - \tmop{var}} ([0, T]
; H^{- \theta_0 + 1})$, which by Markov inequality readily implies tightness
of $u^{\varepsilon}$ needed in the next section.

\subsection{Passage to the limit}\label{s:tight}

With the above results at hand, we are able to perform the final passage to
the limit in the rough path formulation \eqref{eq:urp} of \eqref{eq:u} and
obtain a probabilistically weak rough path solution to \eqref{eq:ulim} in the
sense of Definition~\ref{d:sol2}. This will complete the proof of our main
result, Theorem~\ref{thm:main}.

Even though a number of terms converge directly on any given probability space
where the approximate system is solved, for $u^{\varepsilon}$ we only obtain
uniform estimates implying compactness. Recall also that the bound of
$r^{\varepsilon}$ in $L^2 (0, T ; H^{\gamma})$ and consequently also the time
regularity of $u^{\varepsilon}$ as well as the bound for the remainder
$u^{\varepsilon, \natural}$ are only uniform in $\varepsilon$ after taking
expectations.

We base our compactness argument on Jakubowski--Skorokhod's representation
theorem and change the probability space. We use Jakubowski--Skorokhod's
representation theorem instead of the classical Skorokhod's representation
theorem as one of our function spaces below is not Polish, but falls in the
category of the so-called sub-Polish space where Jakubowski--Skorokhod's
theorem applies, see Section 2.7 in {\cite{BFH18}} for more details.

More precisely, we claim that the above results imply tightness of
$(u^{\varepsilon}, r^{\varepsilon}, Q^{1 / 2} W^{\varepsilon}, Y^{\varepsilon,
1}, Y^{\varepsilon, 2})$ in the path space
\[ \mathcal{X} \assign (C_{\tmop{weak}} ([0, T] ; H) \cap L^2 (0, T ; H))
   \times (L^2 (0, T ; H^{\gamma}), w) \times C ([0, T] ; H) \times
   \mathcal{C}^{\alpha} ([0, T] ; H^{\sigma}), \]
where $(L^2 (0, T ; H^{\gamma}), w)$ denotes $L^2 (0, T ; H^{\gamma})$
equipped with the weak topology. Indeed, the tightness of $u^{\varepsilon}$
follows from \eqref{eq:estu} and \eqref{eq:timeregu} by Lemma~A.2 in
{\cite{MR3918521}}, the tightness of $r^{\varepsilon}$ is a consequence of
\eqref{eq:estr} since bounded sets are relatively compact with respect to the
weak topology, the tightness of $Q^{1 / 2} W^{\varepsilon}$ is immediate since
the law does not depend on $\varepsilon$, and the tightness of the rough path
$(Y^{\varepsilon, 1}, Y^{\varepsilon, 2})$ follows from Theorem~\ref{l:11}.

Accordingly, Jakubowski--Skorokhod's representation theorem yields a
subsequence (still indexed by $\varepsilon$ for notational simplicity) and a
new probability space $(\bar{\Omega}, \bar{\mathcal{F}},
\overline{\tmmathbf{P}})$ with $\mathcal{X}$-valued random variables
$(\bar{u}^{\varepsilon}, \bar{r}^{\varepsilon}, Q^{1 / 2}
\bar{W}^{\varepsilon}, \bar{Y}^{\varepsilon, 1}, \bar{Y}^{\varepsilon, 2})$
and $\left( \bar{u}, \overline{\overline{r}}, Q^{1 / 2} \bar{W}, \bar{Y}^1,
\bar{Y}^2 \right)$ such that $\overline{\tmmathbf{P}}$-a.s.
\[ (\bar{u}^{\varepsilon}, \bar{r}^{\varepsilon}, Q^{1 / 2}
   \bar{W}^{\varepsilon}, \bar{Y}^{\varepsilon, 1}, \bar{Y}^{\varepsilon, 2})
   \rightarrow \left( \bar{u}, \overline{\overline{r}}, Q^{1 / 2} \bar{W},
   \bar{Y}^1, \bar{Y}^2 \right) \qquad \tmop{in} \qquad \mathcal{X}. \]
Then $(\bar{u}^{\varepsilon}, \bar{r}^{\varepsilon}, \bar{Y}^{\varepsilon, 1},
\bar{Y}^{\varepsilon, 2})$ satisfies \eqref{eq:urp}. For a detailed argument
identifying the rough path $(\bar{Y}^{\varepsilon, 1}, \bar{Y}^{\varepsilon,
2})$ we refer to Step 2 in the proof of Proposition~15 in {\cite{FHLN20}}.
Consequently, we obtain also the analogous estimate to \eqref{eq:3.3} for the
corresponding remainder $\bar{u}^{\varepsilon, \natural}$ on the new
probability space. Here $\omega_{\mathbb{A}^{\varepsilon}},
\omega_{\mu^{\varepsilon}}$ are replaced by their counterparts on the new
probability space defined as in \eqref{eq:contr} and \eqref{eq:26a} using
$(\bar{u}^{\varepsilon}, \bar{r}^{\varepsilon}, \bar{Y}^{\varepsilon, 1},
\bar{Y}^{\varepsilon, 2})$ in place of $(u^{\varepsilon}, r^{\varepsilon},
Y^{\varepsilon, 1}, Y^{\varepsilon, 2})$. The result of Theorem~\ref{l:11}
also holds true on the new probability
space.

One significant difference now follows from the $\overline{\tmmathbf{P}}$-a.s.
convergence of $\bar{r}^{\varepsilon}$: in comparison to \eqref{eq:estr} we
gained the following $\overline{\tmmathbf{P}}$-a.s. uniform bound for
$\bar{r}^{\varepsilon}$:
\begin{equation}
  \int_0^{ T} \| \bar{r}^{\varepsilon} \|_{H^{\gamma}}^2 d t \leqslant N
  \label{eq:rnew}
\end{equation}
for some random and $\overline{\tmmathbf{P}}$-a.s. finite constant $N = N
(\bar{\omega}) > 0.$ We employ this fact below to obtain the desired
regularity of the limit remainder.

Now, we can pass to the limit in \eqref{eq:urp}. First, we note that all the
terms except for the remainder $\bar{u}^{\varepsilon, \natural}$ converge
$\overline{\tmmathbf{P}}$-a.s. to the desired limits. More precisely, based on
the strong convergence in $L^2 (0, T ; H)$ of $\bar{u}^{\varepsilon}$ we are
able to pass to the limit in the Stokes as well as the convective term and
combined with the weak convergence of $\bar{r}^{\varepsilon}$ in $L^2 (0, T ;
H^{\gamma})$ we also pass to the limit in the It{\^o}--Stokes drift. As in
Section~\ref{s:av} we obtain that $r^{\varepsilon} \rightarrow \overline{r}$
weakly in $L^2 (0, T ; H^{\gamma})$ $\overline{\tmmathbf{P}}$-a.s., where
$\overline{r}$ is deterministic and time independent, hence $\bar{r} =
\overline{\overline{r}}$. This identifies the It{\^o}--Stokes drift velocity
in \eqref{eq:ulim}. The identification of the driving rough path $(\bar{Y}^1,
\bar{Y}^2)$ as the corresponding lift of $\bar{B} = (- C)^{- 1} Q^{1 / 2}
\bar{W}$ follows as in Theorem~\ref{l:11}.

Accordingly, also $\bar{u}^{\varepsilon, \natural}$ converges
$\overline{\tmmathbf{P}}$-a.s. to some $\bar{u}^{\natural}$ satisfying the
formula from Definition~\ref{d:sol2}. Now, it only remains to prove that
$\bar{u}^{\natural}$ is an honest remainder, namely, it belongs to $C^{p / 3 -
\tmop{var}}_{2, \tmop{loc}} ([0, T] ; H^{- 3})$ $\overline{\tmmathbf{P}}$-a.s.

Recall Remark~\ref{r:11} and note that the analogous claim also holds true on
the new probability space. We denote the random constant analogous to $K$ as
$\bar{K}$ and denote by $\bar{L} \sim 1 / \bar{K}$ the constant determining
the length of admissible intervals in the counterpart of \eqref{eq:3.3} on the
new probability space. Let $I = I (\bar{\omega}) = [\sigma (\bar{\omega}),
\tau (\bar{\omega})] \subset [0, T]$ be an arbitrary random time interval of
length at most $\bar{L} (\omega)$. Then by the analogous estimate to
\eqref{eq:3.3} and \eqref{eq:26a} on the new probability space,
\eqref{eq:estu} and \eqref{eq:rnew}, we have the uniform
$\overline{\tmmathbf{P}}$-a.s. bound
\begin{equation}
  \| \bar{u}^{\varepsilon, \natural} \|_{C_2^{p / 3 - \tmop{var}} (I ; H^{-
  3})} = \bar{\omega}_{\varepsilon, \natural} (\sigma, \tau)^{3 / p} \lesssim
  1 + \bar{\omega}_{\mu^{\varepsilon}} (0, T) \leqslant M, \label{eq:65}
\end{equation}
where $M$ is random and $\overline{\tmmathbf{P}}$-a.s. finite and
$\bar{\omega}_{\varepsilon, \natural}$ and $\bar{\omega}_{\mu^{\varepsilon}}$,
respectively, denote the controls associated to the remainder
$\bar{u}^{\varepsilon, \natural}$ and the drift $\bar{\mu}^{\varepsilon}$
defined as in \eqref{eq:drift} but on the new probability space.

Based on this, we can use lower semicontinuity to conclude that
$\bar{u}^{\natural}$ is $\overline{\tmmathbf{P}}$-a.s. a remainder, namely, it
belongs to $C_{2, \tmop{loc}}^{p / 3 - \tmop{var}} ([0, T] ; H^{- 3})$
$\overline{\tmmathbf{P}}$-a.s. Fix an arbitrary $\bar{\omega}$. Then it holds
for every smooth $\varphi$ and every $s, t \in I (\bar{\omega})$ (all random
variables in the sequel are implicitly evaluated at $\bar{\omega}$)
\[ | \bar{u}^{\natural}_{s t} (\varphi) | = \lim_{\varepsilon \rightarrow 0}
   | \bar{u}_{s t}^{\varepsilon, \natural} (\varphi) | \leqslant \| \varphi
   \|_{H^3} \liminf_{\varepsilon \rightarrow 0} \| \bar{u}^{\varepsilon,
   \natural} \|_{C_2^{p / 3 - \tmop{var}} ([s, t] ; H^{- 3})}, \]
which implies
\[ \| \bar{u}_{s t}^{\natural} \|^{p / 3}_{H^{- 3}} \leqslant
   \liminf_{\varepsilon \rightarrow 0} \| \bar{u}^{\varepsilon, \natural}
   \|^{p / 3}_{C_2^{p / 3 - \tmop{var}} ([s, t] ; H^{- 3})} =
   \liminf_{\varepsilon \rightarrow 0}  \bar{\omega}_{\varepsilon, \natural}
   (s, t) . \]
If $\pi$ is partition of $I (\bar{\omega})$, then by Fatou's lemma,
superadditivity of $\bar{\omega}_{\varepsilon, \natural}$ and \eqref{eq:65},
we obtain
\[ \sum_{(s, t) \in \pi} \| \bar{u}_{s t}^{\natural} \|^{p / 3}_{H^{- 3}}
   \leqslant \liminf_{\varepsilon \rightarrow 0} \sum_{(s, t) \in \pi} 
   \bar{\omega}_{\varepsilon, \natural} (s, t) = \liminf_{\varepsilon
   \rightarrow 0}  \bar{\omega}_{\varepsilon, \natural} (\sigma, \tau)
   \leqslant M, \]
hence finally
\[ \| \bar{u}^{\natural} \|^{p / 3}_{C_2^{p / 3 - \tmop{var}} (I ; H^{- 3})} =
   \sup_{\pi \in \mathcal{P} (I)} \sum_{(s, t) \in \pi} \| \bar{u}_{s
   t}^{\natural} \|^{p / 3}_{H^{- 3}} \leqslant M. \]
Accordingly, we deduce that $\bar{u}^{\natural}$ belongs to $C_{2,
\tmop{loc}}^{p / 3 - \tmop{var}} ([0, T] ; H^{- 3})$
$\overline{\tmmathbf{P}}$-a.s., meaning there is a set $\bar{\Omega}_0 \subset
\bar{\Omega}$ of full probability $\overline{\tmmathbf{P}}$ and a random
covering $(I_k)_{k = 1, \ldots, [T / \bar{L}]}$ of $[0, T]$ such that for
every $\bar{\omega} \in \bar{\Omega}_0$ we have $\bar{u}^{\natural}
(\bar{\omega}) \in C_2^{p / 3 - \tmop{var}} (I_k (\bar{\omega}) ; H^{- 3})$
for every $k = 1, \ldots, [T / \bar{L}]$.

Thus, we have proved that \eqref{eq:ulim} is satisfied by $(\bar{u}, \bar{B})$
in the sense of Definition~\ref{d:sol2}. The proof of Theorem~\ref{thm:main}
is therefore complete.

\

\end{document}